\documentclass[11pt]{amsart}

\usepackage{amsfonts}
\usepackage{amssymb}
\usepackage{amsmath}
\usepackage{amsthm}
\usepackage{stmaryrd} 
\usepackage{wasysym} 
\usepackage[colorlinks,allcolors=magenta]{hyperref}
\usepackage{tikz}

\usepackage[all]{xy}
\usepackage{enumitem}

\setenumerate{topsep=0pt,labelwidth=0pt,align=left,labelindent=0pt,labelsep=0in,labelwidth=0.275in,leftmargin=0.275in,label=\rm(\alph*)}

\renewcommand{\phi}{\varphi}
\newcommand{\eps}{\epsilon}
\renewcommand{\hat}{\widehat}
\renewcommand{\tilde}{\widetilde}
\renewcommand{\bar}{\overline}
\newcommand{\restr}{\upharpoonright}
\newcommand{\forces}{\mathrel{\Vdash}}

\newcommand{\N}{\mathbb{N}}
\newcommand{\Q}{\mathbb{Q}}
\newcommand{\R}{\mathbb{R}}
\newcommand{\C}{\mathbb{C}}

\newcommand{\A}{\mathbb{A}}

\renewcommand{\P}{\mathbb{P}}

\newcommand{\LL}{\mathcal{L}}

\newcommand{\LF}{\mathcal{F}}
\newcommand{\LP}{\mathcal{P}}
\newcommand{\LC}{\mathcal{C}}

\newcommand{\LA}{\mathcal{A}}
\newcommand{\LB}{\mathcal{B}}
\newcommand{\LU}{\mathcal{U}}
\newcommand{\LS}{\mathcal{S}}

\newcommand{\LD}{\mathcal{D}}
\newcommand{\LG}{\mathcal{G}}

\renewcommand{\t}{\mathfrak{t}}

\renewcommand{\d}{\mathfrak{d}}

\DeclareMathOperator{\ran}{ran}

\newtheorem{thm}{Theorem}[section]
\newtheorem{prop}[thm]{Proposition}
\newtheorem{lemma}[thm]{Lemma}
\newtheorem{cor}[thm]{Corollary}
\newtheorem*{nonumthm}{Theorem}
\newtheorem*{problem}{Problem}
\newtheorem*{conj}{Conjecture}

\theoremstyle{definition}
\newtheorem{defn}[thm]{Definition}
\newtheorem{example}[thm]{Example}

\theoremstyle{remark}

\newtheorem*{claim}{Claim}
\newtheorem*{ques}{Question}

\newenvironment{proofclaim}[1][Proof of claim.]{\begin{proof}[#1]}{\end{proof}}

\newcommand{\LH}{\mathcal{H}}
\newcommand{\LK}{\mathcal{K}}

\newcommand{\ess}{\mathrm{ess}}

\newcommand{\supp}{\mathrm{supp}}

\DeclareMathOperator{\linspan}{span}

\newcommand{\concat}{^\smallfrown}

\newcommand{\ZFC}{\mathsf{ZFC}}
\newcommand{\CH}{\mathsf{CH}}
\newcommand{\MA}{\mathsf{MA}}

\newcommand{\D}{\mathbb{D}}
\newcommand{\LE}{\mathcal{E}}

\newcommand{\bb}{\mathrm{bb}}

\renewcommand{\P}{\mathbb{P}}
\newcommand{\G}{\mathbb{G}}

\newcommand{\V}{\mathbf{V}}
\renewcommand{\L}{\mathbf{L}}

\newcommand{\FIN}{\mathrm{FIN}}

\newcommand{\osc}{\mathrm{osc}}

\title{A local Ramsey theory for block sequences}
\date{August 20, 2018}
\author{Iian B. Smythe}
\address{Department of Mathematics, Rutgers, The State University of New Jersey, Piscataway, NJ, USA 08854}
\email{i.smythe@rutgers.edu}


\thanks{The author is partially supported by NSERC award PGSD2-453779-2014 and NSF grant DMS-1600635. He would also like to thank his PhD advisor, Justin Tatch Moore, for continued guidance and suggesting the problem of characterizing $\L(\R)$-generic filters for the projections in the Calkin algebra which motivated this work.}

\begin{document}

\begin{abstract}
	We develop local forms of Ramsey-theoretic dichotomies for block sequences in infinite-dimensional vector spaces, analogous to Mathias' selective coideal form of Silver's theorem for analytic partitions of $[\N]^\infty$. Under large cardinals, these results are extended to partitions in $\L(\R)$ and $\L(\R)$-generic filters of block sequences are characterized. Variants of these results are also established for block sequences in Banach spaces and for projections in the Calkin algebra.
\end{abstract}

\maketitle

\section{Introduction}\label{sec:1}

Ramsey-theoretic techniques have a long history of use in Banach space theory, see e.g., \cite{MR2145246}. Most relevant for the present work is Gowers' dichotomy for infinite block sequences in Banach spaces: 

\begin{nonumthm}[Gowers \cite{MR1421876} \cite{MR1954235}]
	Let $B$ be an infinite-dimensional Banach space with a Schauder basis. If $\A$ is an analytic set of normalized block sequences, then for any $\Delta>0$, there is a block sequence $Y$ such that either
	\begin{enumerate}[label=\rm{(\roman*)}]
		\item every normalized block subsequence of $Y$ is in $\A^c$, or
		\item II has a strategy in the Gowers game $G^*[Y]$ for playing into $\A_{\Delta}$.
	\end{enumerate}
\end{nonumthm}

Loosely speaking, this result says that for $\A$ as described, there is a block sequence $Y$ such that either all of $Y$'s normalized block subsequences are disjoint from $\A$, or there is a wealth of block subsequences of $Y$ which are within a small perturbation of $\A$. This was used, together with work of Komorowski and Tomczak-Jaegerman \cite{MR1324462}, to solve (affirmatively) the homogeneous space problem.

In the setting of a discrete countably infinite-dimensional vector space $E$ over a countable field, Rosendal isolated an ``exact'' version of Gowers' dichotomy which yields a much simplified proof of the original result:

\begin{nonumthm}[Rosendal \cite{MR2604856}]
	If $\A$ is an analytic set of block sequences in $E$, then there is a block sequence $Y$ such that either
	\begin{enumerate}[label=\rm{(\roman*)}]
		\item I has a strategy in the infinite asymptotic game $F[Y]$ for playing into $\A^c$, or
		\item II has a strategy in the Gowers game $G[Y]$ for playing into $\A$.
	\end{enumerate}
\end{nonumthm}

These dichotomies are analogues, in the Banach space and vector space settings, respectively, of the following result for partitions of $[\N]^\infty$, the set of infinite subsets of the natural numbers.

\begin{nonumthm}[Silver \cite{MR0332480}]
	If $\A\subseteq[\N]^\infty$ is analytic, then there is a $y\in[\N]^\infty$ with either all of its further infinite subsets disjoint from, or contained in, $\A$.
\end{nonumthm}

While the theory of topological Ramsey spaces, in the sense of \cite{MR2603812}, encompasses many variations on this result, the dichotomies of Gowers and Rosendal highlighted above do not fall into this framework.

An important generalization of Silver's theorem is the following ``local'' Ramsey theorem, showing that the witness $y$ in the conclusion can always be found in a given selective coideal (or ``happy family''):

\begin{nonumthm}[Mathias \cite{MR0491197}]
	Let $\LH\subseteq[\N]^\infty$ be a selective coideal. If $\A\subseteq[\N]^\infty$ is analytic, then there is a $y\in\LH$ with either all of its further infinite subsets disjoint from, or contained in, $\A$.
\end{nonumthm}

By passing to a forcing extension resulting from the L\'evy collapse of a Mahlo cardinal, Mathias extended these results to all partitions $\A$ which are ``reasonably definable", that is, in the definable closure of the reals $\L(\R)$. Later work of Farah and Todorcevic \cite{MR1644345} generalized this to semiselective coideals and showed that under stronger large cardinal hypotheses the passage to a forcing extension is not necessary. The extension of Silver's theorem to all partitions in $\L(\R)$ is due to Shelah and Woodin \cite{MR1074499}. Similar results have been developed recently for topological Ramsey spaces \cite{MR2330595} \cite{local_Ramsey}.

The upshot of obtaining these local results is two-fold: We clearly isolate the combinatorial properties which enable the original dichotomies, and we obtain greater control over the witnesses to said dichotomies.

This latter point was used by Todorcevic \cite{MR1644345} to characterize, under large cardinal hypotheses, selective ultrafilters as being exactly those which are generic for $([\N]^\infty,\subseteq^*)$ over $\L(\R)$. Such ultrafilters are said to possess ``complete combinatorics'', following Blass and Laflamme \cite{MR996504} who used this phrase to describe ultrafilters which are generic over $\L(\R)$ after collapsing a Mahlo cardinal. We instead ask for genericity over $\L(\R)$ of the ground model, at the expense of stronger large cardinal hypotheses.

Using \cite{MR2604856} as a starting point, we develop local versions of Gowers' and Rosendal's dichotomies. When $E$ is a countably infinite-dimensional space with basis $(e_n)$ over some countable field $F$, we isolate in \S\ref{sec:2} \emph{$(p^+)$-families} of block sequences, collections of block sequences closed under certain diagonalizations and witnessing a weak pigeonhole principle, and in \S\ref{sec:3} establish our local form of Rosendal's dichotomy:

\begin{thm}\label{thm:local_Rosendal}
	Let $\LH$ be a $(p^+)$-family of block sequences in $E$. If $\A$ is an analytic set of block sequences and $X\in\LH$, then there is a $Y\in\LH\restr X$ such that either
	\begin{enumerate}[label=\rm{(\roman*)}]
		\item I has a strategy in $F[Y]$ for playing into $\A^c$, or
		\item II has a strategy in $G[Y]$ for playing into $\A$.
	\end{enumerate}
\end{thm}

Stronger properties of families are discussed in \S\ref{sec:4}, notably \emph{strategic} families. The existence of filters with these properties is considered in \S\ref{sec:5} and \S\ref{sec:6}, where their existence is proved to be independent of $\ZFC$.

In \S\ref{sec:7} we show that, under large cardinal hypothesis, strategic $(p^+)$-filters have complete combinatorics for infinite block sequences with the block subsequence ordering, and generalize Theorem \ref{thm:local_Rosendal} to partitions in $\L(\R)$ (the corresponding extension of Gowers' original result is due to L\'{o}pez-Abad \cite{MR2179777}, see also \cite{MR1873008}). This requires an analysis of a Mathias-like notion of forcing used to build generic block sequences.

\begin{thm}\label{thm:gen_over_L(R)}
	Assume that there is a supercompact cardinal. A filter $\LG$ of block sequences in $E$ is $\L(\R)$-generic for the partial ordering of block sequences if and only if it is a strategic $(p^+)$-filter.
\end{thm}

\begin{thm}\label{thm:local_Rosendal_L(R)}
	Assume that there is a supercompact cardinal. Let $\LH$ be a strategic $(p^+)$-family of block sequences in $E$. If $\A$ is a set of block sequences in $\L(\R)$ and $X\in\LH$, then there is a $Y\in\LH\restr X$ such that either
	\begin{enumerate}[label=\rm{(\roman*)}]
		\item I has a strategy in $F[Y]$ for playing into $\A^c$, or
		\item II has a strategy in $G[Y]$ for playing into $\A$.
	\end{enumerate}	
\end{thm}

In \S\ref{sec:8} we consider normed vector spaces and Banach spaces. For an infinite-dimensional separable Banach space $B$ with a Schauder basis, we develop the notion of \emph{spread $(p^*)$-families}, similar to the $(p^+)$-families in \S\ref{sec:2}, and establish the following local form of Gowers' dichotomy and its extension to $\L(\R)$:

\begin{thm}\label{thm:local_Gowers}
	Let $\LH$ be a spread $(p^*)$-family of normalized block sequences in $B$ which is invariant under small perturbations. If $\A$ is an analytic set of normalized block sequences and $X\in\LH$, then for any $\Delta>0$, there is a $Y\in\LH\restr X$ such that either
	\begin{enumerate}[label=\rm{(\roman*)}]
		\item every normalized block subsequence of $Y$ is in $\A^c$, or
		\item II has a strategy in $G^*[Y]$ for playing into $\A_{\Delta}$.
	\end{enumerate}
\end{thm}

\begin{thm}\label{thm:local_Gowers_L(R)}
	Assume that there is a supercompact cardinal. Let $\LH$ be a strategic $(p^*)$-family of normalized block sequences in $B$ which is invariant under small perturbations. If $\A$ is a set of normalized block sequences in $\L(\R)$ and $X\in\LH$, then for any $\Delta>0$, there is a $Y\in\LH\restr X$ such that either
	\begin{enumerate}[label=\rm{(\roman*)}]
		\item every normalized block subsequence of $Y$ is in $\A^c$, or
		\item II has a strategy in $G^*[Y]$ for playing into $\A_{\Delta}$.
	\end{enumerate}
\end{thm}

It is our hope that Theorem \ref{thm:local_Gowers} will afford new applications of the techniques introduced by Gowers in \cite{MR1954235} to obtain block sequences in Banach spaces with simultaneous properties, some captured by the target set $\A$, while others by the family $\LH$.

In \S\ref{sec:9} we apply these results to the study of the projections in the Calkin algebra, the quotient of the bounded operators $\LB(H)$ on a Hilbert space $H$ by the compact operators. The natural ordering on projections in the Calkin algebra induces an ordering $\leq_\ess$ on $\LP_\infty(H)$, the infinite-rank projections in $\LB(H)$. We give a version of Theorem \ref{thm:gen_over_L(R)} for filters in this ordering:

\begin{thm}\label{thm:Calkin_L(R)_gen}
	Assume that there is a supercompact cardinal. A filter $\LG$ in $(\LP_\infty(H),\leq_\ess)$ is $\L(\R)$-generic if and only if projections onto block subspaces are $\leq_\ess$-dense in $\LG$ and the associated family of block sequences in $H$ is a strategic $(p^*)$-family.
\end{thm}

Generic filters for $(\LP_\infty(H),\leq_\ess)$ induce \emph{pure states} on $\LB(H)$, via the theory of \emph{quantum filters} introduced by Farah and Weaver \cite{FarahAST}. It is known that these generic pure states are not pure on any atomic maximal abelian self-adjoint subalgebra (essentially due to Farah and Weaver \cite{FarahAST}), and are thus counterexamples to a conjecture of Anderson \cite{MR672813}. We show that any family satisfying the hypotheses of Theorem \ref{thm:local_Gowers} and generating a pure state on $\LB(H)$ produces such a counterexample. We caution that our counterexamples remain beyond $\ZFC$.

\begin{thm}\label{thm:pure_states}
	A spread $(p^*)$-family $\LH$ of block sequences in $H$ which is $\leq_\ess$-centered induces a singular pure state $\rho$ on $\LB(H)$ which is not pure on any atomic maximal abelian self-adjoint subalgebra.
\end{thm}

\S\ref{sec:last} concludes the paper with questions for future investigation.

An effort has been made to keep the set-theoretic prerequisites for understanding this work to a minimum with the hope that the material, particularly in \S\ref{sec:3} and \S\ref{sec:8}, may be used for further applications in Banach space and operator theory. We assume a familiarity with the basic properties of Polish spaces, Borel sets, and analytic sets (as covered in \cite{MR1321597}) throughout. We only make explicit use of the method of forcing and large cardinal hypotheses in \S\ref{sec:5} and \S\ref{sec:7}, with occasional reference back to that material in \S\ref{sec:8} and \S\ref{sec:9}. The Banach space prerequisites amount to little more than a familiarity with basic sequences (as covered in the first sections of \cite{MR2192298}).

\section{Families of block sequences}\label{sec:2}

Fix a countable field $F$, a countably infinite-dimensional $F$-vector space $E$, and an Hamel $F$-basis $(e_n)$ for $E$. Typically we will think of $F$ as a subfield of $\C$, but this is not necessary; $F$ may even be finite. Given $v\in E$, say with $v=\sum_{n=0}^Na_ne_n$, let $\supp(v)=\{n\in\N:a_n\neq 0\}$, the \emph{support} of $v$. We write $n<v$ if $n<\min(\supp(v))$ and $v<w$ if $\max(\supp(v))<\min(\supp(w))$. 

We say that a (finite or infinite) sequence $(x_n)$ of non-zero vectors in $E$ is a \emph{block sequence} (with respect to $(e_n)$) if for all $n$, $x_n<x_{n+1}$. If $\vec{x}=(x_0,\ldots,x_n)$ is a finite block sequence, let $\supp(\vec{x})=\bigcup_{i=0}^n\supp(x_i)$, and for $X$ any block sequence, let $\langle X\rangle=\linspan(X)\setminus\{0\}$. We will abuse notation and write $E$ for $E\setminus\{0\}$, and use ``vector'' to mean non-zero vector.

Let $\bb^\infty(E)$ be the collection of all infinite block sequences in $E$, which we consider as a subspace of $E^\N$, where $E$ has the discrete topology. It is easy to check that $\bb^\infty(E)$ is a $G_\delta$ subset of $E^\N$, and thus a Polish space. Let $\bb^{<\infty}(E)$ be the collection of all finite block sequences in $E$.

For $X=(x_n)$ and $Y=(y_n)$ in $\bb^\infty(E)$, we write $X\preceq Y$ if $(x_n)$ is a block sequence with respect to $(y_n)$, sometimes called a \emph{block subsequence} of $Y$, or equivalently (for \emph{block} sequences), $\langle X\rangle\subseteq\langle Y \rangle$. We write $X\preceq^* Y$ if for some $m$, $X/m\preceq Y$, where $X/m$ is the tail of $X$ with supports above $m$. For $\vec{x}\in\bb^{<\infty}(E)$, write $X/\vec{x}$ for $X/\max(\supp(\vec{x}))$. Note that the orderings $\preceq$ and $\preceq^*$ fail to be antisymmetric, but are reflexive and transitive.

We will make repeated use of the following order-theoretic notions: A subset $D$ of a pre-order $(P,\leq)$ (that is, $\leq$ is reflexive and transitive) is \emph{dense} if for all $p\in P$, there is a $q\in D$ with $q\leq p$. It is, moreover, \emph{dense open}, if whenever $q\leq p\in D$, then $q\in D$. Elements $p$ and $q$ in $P$ are \emph{compatible} if they have a common lower bound in $P$, and \emph{incompatible} otherwise.

Compatibility in $(\bb^\infty(E),\preceq)$ is equivalent to that in $(\bb^\infty(E),\preceq^*)$ and we write $X\bot Y$ when $X$ and $Y$ are incompatible. The following observation shows that $(\bb^\infty(E),\preceq)$ can be identified with a dense suborder of the lattice of all infinite-dimensional subspaces of $E$. In particular, $X$ and $Y$ are compatible if and only if $\langle X\rangle\cap\langle Y\rangle$ is infinite-dimensional.

\begin{lemma}\label{lem:subspaces_cont_blocks}
	If $X$ is an infinite-dimensional subspace of $E$, then $X$ contains an infinite block sequence.	
\end{lemma}

\begin{proof}
	By taking appropriate linear combinations, one can show that for any $N$, $X$ contains an infinite-dimensional subspace whose supports are above $N$. From this, it is easy to inductively construct a block sequence in $X$.
\end{proof}

Throughout, when we speak of a \emph{family} $\LH\subseteq\bb^\infty(E)$, we mean a non-empty subset which is closed upwards with respect to $\preceq^*$. For $X\in\LH$, we denote by $\LH\restr X=\{Y\in\LH:Y\preceq X\}$. A \emph{filter} $\LF\subseteq\bb^\infty(E)$ is a family such that for every $X,Y\in\LF$, there is a $Z\in\LF$ with $Z\preceq X$ and $Z\preceq Y$.

\begin{defn}
	\begin{enumerate}
		\item Given a descending sequence $X_0\succeq X_1\succeq \cdots$ in $\bb^\infty(E)$, we call $Y\in\bb^\infty(E)$ a \emph{diagonalization} of $(X_n)$ if for all $n$, $Y \preceq^* X_n$.
		\item Given a sequence $(\LD_n)$ of subsets of $\bb^\infty(E)$, we call $Y$ a \emph{diagonalization} of $(\LD_n)$ if for each $n$, there is an $X_n\in\LD_n$ such that $Y\preceq^* X_n$.
	\end{enumerate}
\end{defn}

For $\LH\subseteq\bb^\infty(E)$, a set $\LD$ is \emph{$\preceq$-dense (open) in $\LH$} if $\LD\cap\LH$ is.

\begin{defn}
	 A family $\LH\subseteq\bb^\infty(E)$ is a \emph{$(p)$-family}, or has the \emph{$(p)$-property}, if whenever $X_0\succeq X_1\succeq \cdots$ is a decreasing sequence with each $X_n\in\LH$, there is a diagonalization $Y\in\LH$ of $(X_n)$.
\end{defn}

It is easy to see that $\bb^\infty(E)$ itself is a $(p)$-family. We note that every $(p)$-family $\LH$ contains a diagonalization of any given sequence $(\LD_n)$ of $\preceq$-dense open subsets in $\LH$: build a decreasing sequence $(X_n)$ in $\LH$ with each $X_n\in\LD_n$, then any diagonalization $Y\in\LH$ of $(X_n)$ will be a diagonalization of $(\LD_n)$. This can be done below any given $X\in\LH$, so the set of such diagonalizations is $\preceq$-dense in $\LH$. This latter property, which could be called the \emph{weak $(p)$-property}, will be sufficient for all of the results in \S\ref{sec:3}, and in particular, for Theorem \ref{thm:local_Rosendal}.

Recall that $\LH\subseteq[\N]^\infty$ is a \emph{coideal} if it contains all co-finite sets, is closed upwards with respect to $\subseteq$, and whenever $Y_0\cup Y_1\in\LH$, then one of $Y_0$ or $Y_1$ is also in $\LH$. This last property asserts that $\LH$ witnesses the pigeonhole principle. In our setting, provided $|F|>2$,\footnote{When $|F|=2$, such a pigeonhole principle for block subspaces does hold; this is essentially Hindman's Theorem \cite{MR0349574}.} the ``obvious'' formulation of the pigeonhole principle is simply false, as the following example shows:

\begin{example}\label{ex:asym_sets}\hspace{-0.5em}\footnote{The author would like to thank Jordi L\'opez-Abad for pointing out this example which has the advantage of being well-defined at the level of the spanned subspaces.}
	Consider the case when $F\subseteq\R$. Similar examples can be constructed whenever $|F|>2$, cf.~Theorem 7 in \cite{MR2737185}. For a vector $x\in E$ define the \emph{oscillation} $\osc(x)$ as the number of times the sign of the non-zero coefficients of $x$ alternate in its expansion with respect to $(e_n)$. So, $\osc(e_0-e_1+e_2) = 2$, $\osc(e_2+e_4-e_5+e_7-e_{10})=3$, etc.
	
	Define $A_0\subseteq E$ (respectively, $A_1\subseteq E$) to be the set of all $x\in E$ such that $\osc(x)$ is even (respectively, odd), and let $\A_i=\{(x_n):x_0\in A_i\}$ for $i=0,1$. The $\A_i$ are clopen sets which partition $\bb^\infty(E)$. Moreover, the pair $\A_0$, $\A_1$ is \emph{asymptotic}, that is, for any $X\in\bb^\infty(E)$ and $i=0,1$, there is $Y_i\preceq X$ such that $Y_i\in\A_i$. To see this, suppose that $X=(x_n)$ is such that $X\in\A_0$, so $\osc(x_0)$ is even. If $\osc(x_1)$ is odd, then $(x_n)_{n\geq 1}\preceq X$ and in $\A_1$. If $\osc(x_1)$ is even, then let $x=x_0-x_1$ if the signs of the last non-zero coefficient in $x_0$ and the first in $x_1$ agree, and $x=x_0+x_1$ otherwise. In either case, $\osc(x)=\osc(x_0)+\osc(x_1)+1$, so $(x,x_2,x_3,\ldots)$ is in $\A_1$.
\end{example}

The following is a weak analogue of the pigeonhole property of coideals.

\begin{defn}
	Let $\LH\subseteq\bb^\infty(E)$ be a family.
	\begin{enumerate}
		\item A subset $\LD\subseteq\bb^\infty(E)$ is \emph{$\LH$-dense} below some $X\in\LH$ if for every $Y\in\LH\restr X$, there is a $Z\preceq Y$ with $Z\in\LD$. A set $D\subseteq E$ is $\LH$-dense below $X$ if $\{Z:\langle Z\rangle\subseteq D\}$ is.
		\item $\LH$ is \emph{full} if whenever $D\subseteq E$ (not necessarily a subspace) and $X\in\LH$ are such that $D$ is $\LH$-dense below $X$, there is a $Z\in\LH\restr X$ with $\langle Z\rangle\subseteq D$.	
	\end{enumerate}
\end{defn}

Fullness allows one to upgrade $\{Z:\langle Z\rangle\subseteq D\}$ being $\LH$-dense below $X$ to being $\preceq$-dense (open) below $X$ in $\LH$. Obviously $\bb^\infty(E)$ itself is a full family. If the family in question is a filter $\LF$, we may simplify the definition of fullness by replacing $X$ with $(e_n)$ (or any element of $\LF$). We note that any full filter is maximal; this can be seen by applying the definition of fullness when $D$ is a block subspace. It is shown in Proposition \ref{prop:full_necessary} that fullness is necessary for Theorem \ref{thm:local_Rosendal}.

\begin{defn}
	A family in $\bb^\infty(E)$ which is full and has the $(p)$-property will be called a \emph{$(p^+)$-family}. Likewise for \emph{$(p^+)$-filter}. 
\end{defn}

\begin{lemma}\label{lem:dense_sets_diag_full}
	\begin{enumerate}
		\item For $X_0\succeq X_1\succeq \cdots$ in $\bb^\infty(E)$, the set
		\[
			\LD_{(X_n)}=\{Y:Y\text{ is a diagonalization of } (X_n)\text{ or }\exists n(Y\bot X_n)\}
		\]
		is $\preceq$-dense open.
		
		\item For $D\subseteq E$ and $X\in\bb^\infty(E)$, the set
		\[
			\LD_{D,X}=\{Z:\langle Z\rangle \subseteq D \text{ or } \forall V\preceq X(\langle V\rangle\subseteq D \rightarrow V\bot Z)\}
		\]
		is $\preceq$-dense open below $X$.
	\end{enumerate}
\end{lemma}

\begin{proof}
	(a)	Take $Y\in\bb^\infty(E)$ which is compatible with all of the $X_n$. We can build a diagonalization $X=(x_n)\preceq Y$ by picking vectors $x_n\in\langle X_n\rangle\cap\langle Y\rangle$ with $x_n<x_{n+1}$.

	\noindent(b) Take $Y\preceq X$. If there is no $Z\preceq Y$ such that $\langle Z\rangle\subseteq D$, then for any $V\preceq X$ with $\langle V\rangle\subseteq D$, it must be that $V\bot Y$, as otherwise any $Z$ witnessing the compatibility of $V$ and $Y$ would satisfy $\langle Z\rangle\subseteq D$.
\end{proof}

Lemma \ref{lem:dense_sets_diag_full} will be used to construct $(p^+)$-filters in \S5. We will see in Corollary \ref{cor:full_filters_ind} that the existence of full filters is independent of $\ZFC$.

\section{Games with vectors and a local Rosendal dichotomy}\label{sec:3}

The \emph{Gowers game} played below $X\in\bb^\infty(E)$, denoted $G[X]$, is defined as follows: Two players, I and II, alternate with I going first and playing block sequences $X_k\preceq X$, and II responding with vectors $y_k\in \langle X_k\rangle$ subject to the constraint $y_k<y_{k+1}$. The block sequence $(y_k)$ is the \emph{outcome} of a play of the game. Given $\vec{x}\in\bb^{<\infty}(E)$ and $X\in\bb^\infty(E)$, the game $G[\vec{x},X]$ is defined exactly as $G[X]$ except that II is restricted to playing vectors above $\vec{x}$ and the outcome is $\vec{x}\concat(y_k)$. This is a discrete version of the game defined by Gowers in \cite{MR1421876} \cite{MR1954235}

A \emph{strategy} for II in $G[\vec{x},X]$ is a function $\alpha$ taking sequences $(X_0,\ldots,X_k)$ of possible prior moves by I to vectors $y\in \langle X_{k}\rangle$, with $\vec{x}<\alpha(X_0,\ldots,X_{k-1})<y$, for all $k$. Given a set $\A\subseteq\bb^\infty(E)$, we say that $\alpha$ is a strategy in $G[\vec{x},X]$ for \emph{playing into $\A$} if whenever II follows $\alpha$ (that is, at each turn, given as input I's prior moves, they play the output of $\alpha$), the resulting outcome lies in $\A$. These notions are defined likewise for I.

The \emph{infinite asymptotic game} \cite{MR2566964} \cite{MR2604856} played below $X$, denoted $F[X]$, is defined in a similar fashion: Two players, I and II, alternate with I going first and playing natural numbers $n_k$, and II responding with vectors $y_k\in\langle X/n_k\rangle$ subject to the constraint $y_k<y_{k+1}$. Again, $(y_k)$ is the \emph{outcome} of a play of the game. The game $F[\vec{x},X]$ is defined as above, as are \emph{strategies} for I and II, and the notion of having a strategy for \emph{playing into} a set.

It is important to note that plays of $F[\vec{x},X]$ can be considered as plays of $G[\vec{x},X]$ where I is restricted to playing tail block subsequences of $X$. Consequently, if II has a strategy in $G[\vec{x},X]$ for playing into a set $\A$, then II has such a strategy in $F[\vec{x},X]$ as well. Similarly, if I has a strategy in $F[\vec{x},X]$ for playing into $\A$, then they have such a strategy in $G[\vec{x},X]$. 

The following generalizes the notion of \emph{strategically Ramsey} given in \cite{MR2604856}, where $\LH$ was taken to be all of $\bb^\infty(E)$.

\begin{defn}	
	For $\LH\subseteq\bb^\infty(E)$ a family, we say that a subset $\A\subseteq \bb^\infty(E)$ is \emph{$\LH$-strategically Ramsey} if for all $\vec{y}\in\bb^{<\infty}(E)$ and $X\in\LH$, there is a $Y\in\LH\restr X$ such that either
	\begin{enumerate}[label=\rm{(\roman*)}]
		\item I has a strategy in $F[\vec{y},Y]$ for playing into $\A^c$, or
		\item II has a strategy in $G[\vec{y},Y]$ for playing into $\A$.
	\end{enumerate}
\end{defn}

Note that consequences (i) and (ii) are mutually exclusive by our comments above. The key fact about $\LH$-strategically Ramsey sets is that the witness, $Y$ in the above definition, can be found in $\LH$.

Our goal for the remainder of this section is to outline the proof that, for any $(p^+)$-family $\LH$, analytic sets are $\LH$-strategically Ramsey, thereby establishing Theorem \ref{thm:local_Rosendal}. Much of what follows closely hews to \cite{MR2604856}, and is a variation on the combinatorial forcing technique used in \cite{MR2603812}.

\begin{defn}\label{def:good_bad_worse}
	Let $\LH$ be a family and $\A\subseteq\bb^\infty(E)$ be given. For $\vec{y}\in\bb^{<\infty}(E)$ and $Y\in\LH$, we say that
		\begin{enumerate}[label=\rm{(\arabic*)}]
			\item $(\vec{y},Y)$ is \emph{good (for $\A$)} if II has a strategy in $G[\vec{y},Y]$ for playing into $\A$,
			\item $(\vec{y},Y)$ is \emph{bad (for $\A$)} if for all $Z\in\LH\restr Y$, $(\vec{y},Z)$ is not good.
			\item $(\vec{y},Y)$ is \emph{worse (for $\A$)} if it is bad and there is an $n$ such that for every $v\in \langle Y/n\rangle$, $(\vec{y}\concat v,Y)$ is bad.
		\end{enumerate}
\end{defn}

Reference to $\A$ and $\LH$ will be suppressed where understood.

\begin{lemma}\label{lem:good_worse}
	If $\LH$ is a $(p^+)$-family and $\A\subseteq\bb^\infty(E)$, then for every $\vec{x}\in\bb^{<\infty}(E)$ and $X\in\LH$, there is a $Y\in\LH\restr X$ such that either
	\begin{enumerate}[label=\rm{(\roman*)}]
		\item $(\vec{x},Y)$ is good, or
		\item I has a strategy in $F[\vec{x},Y]$ for playing into
		\[
			\{(z_n):\forall n(\vec{x}\concat(z_0,\ldots,z_n),Y)\text{ is worse}\}.
		\]
	\end{enumerate}
\end{lemma}

\begin{proof}
	Observe that if $(\vec{y},Y)$ is good/bad/worse and $Z\preceq^*Y$ is in $\LH$, then $(\vec{y},Z)$ is also good/bad/worse. It is immediate that for each $\vec{y}$, the set
	\[
		\LD_{\vec{y}}=\{Y\in\LH:(\vec{y},Y)\text{ is either good or bad}\}
	\]
	is $\preceq$-dense open in $\LH$.	

	\begin{claim}
		If $(\vec{y},Y)$ is bad, then for all $Z\in\LH\restr Y$, there is a $V\preceq Z$ such that for all $x\in\langle V/\vec{y}\rangle$, $(\vec{y}\concat x, Y)$ is not good. 
	\end{claim}
	
	\begin{proofclaim}
		Let $(\vec{y}, Y)$ be bad. Towards a contradiction, suppose that there is some $Z\in\LH\restr Y$ such that for all $V\preceq Z$, there is an $x\in\langle V/\vec{y}\rangle$ such that $(\vec{y}\concat x,Y)$ is good. We claim that $(\vec{y},Z)$ is good. If I plays $V\preceq Z$, then by supposition there is some $x\in\langle V/\vec{y}\rangle$ such that $(\vec{y}\concat x,Z)$ is good. Let II play that $x$ and from then on follow the strategy given from $(\vec{y}\concat x,Z)$ being good. This is contrary to $(\vec{y},Y)$ being bad.
	\end{proofclaim}
	
	\begin{claim}
		For each $\vec{y}$, the set
		\[
			\LE_{\vec{y}}=\{Z\in\LH:(\vec{y},Z)\text{ is either good or worse}\}
		\]
		is $\preceq$-dense open in $\LH$.	
	\end{claim}
	
	\begin{proofclaim}
		Fix $\vec{y}$ and let $Y\in\LH$. Since the sets $\LD_{\vec{x}}$ are dense in $\LH$ and there are only countably many $\vec{x}$, the $(p)$-property allows us to diagonalize all of them within $\LH$ and assume that for all $\vec{x}$, $(\vec{x},Y)$ is either good or bad. Suppose that $(\vec{y},Y)$ is bad. Let $D=\{x:(\vec{y}\concat x,Y)\text{ is not good}\}$. By the previous claim, $D$ is $\LH$-dense below $Y$. Since $\LH$ is full, there is a $Z\in\LH\restr Y$ such that $\langle Z\rangle\subseteq D$. If $z\in\langle Z\rangle$, then $(\vec{y}\concat z,Z)$ is not good, hence bad, by our choice of $Y$. Thus, $(\vec{y},Z)$ is worse.
	\end{proofclaim}

	We can now prove the lemma. By the previous claim, we have a $Y\in\LH\restr X$ so that for all $\vec{y}$, $(\vec{x}\concat\vec{y},Y)$ is either good or worse. If $(\vec{x},Y)$ is good, we're done, so suppose that $(\vec{x}, Y)$ is worse. We will describe a strategy for I in $F[\vec{x},Y]$: Suppose that at some point in the game $(z_0,\ldots,z_k)$ has been played by II so that $(\vec{x}\concat(z_0,\ldots,z_k),Y)$ is worse. Then, there is some $n$ such that for all $z\in \langle Y\rangle$, if $n<z$, then $(\vec{x}\concat(z_0,\ldots,z_k)\concat z,Y)$ is bad, hence worse. Let I play $n$.
\end{proof}

\begin{lemma}[cf. Lemma 2 in \cite{MR2604856}]\label{lem:local_Rosendal_open}
	Let $\LH\subseteq\bb^\infty(E)$ a $(p^+)$-family. Then, open sets are $\LH$-strategically Ramsey.
\end{lemma}

\begin{proof}
	Let $\A\subseteq\bb^\infty(E)$ be open. Given $\vec{x}\in\bb^{<\infty}(E)$ and $X\in\LH$, by Lemma \ref{lem:good_worse}, there is a $Y\in\LH\restr X$ such that either $(\vec{x},Y)$ is good, in which case we're done, or I has a strategy in $F[\vec{x},Y]$ to play $(z_n)$ such that for all $n$, $(\vec{x}\concat(z_0,\ldots,z_n),Y)$ is worse. In the latter case, if I follows this strategy, as II builds $(z_n)$, for no $m$ can II have a strategy in $G[\vec{x}\concat(z_0,\ldots,z_m),Y]$ to play in $\A$. Since $\A$ is open, this means that $\vec{x}\concat(z_0,z_1,\ldots)\notin\A$ and I has a strategy for playing into $\A^c$.
\end{proof}

\begin{proof}[Proof (sketch) of Theorem \ref{thm:local_Rosendal}.]
	The proof closely follows that of Theorem 5 in \cite{MR2604856}, where $\LH=\bb^\infty(E)$. The idea of the proof is, given a Souslin scheme $\{\A_s\}_{s\in\N^{<\infty}}$ for an analytic set $\A$, we can use Lemma \ref{lem:local_Rosendal_open} and diagonalization to find a $Y\in\bb^\infty(E)$ such that if I does not have a strategy in $F[Y]$ for playing into $\A^c$, then in $G[Y]$, II can build a sequence $(z_k)$ such that I continues to have no strategy in $F[(z_0,\ldots,z_k),Y]$ for playing into $\A_s^c$, where $s$ is an initial segment of some branch $y$ in $\N^\N$. $y$ will witness that II's strategy has produced an outcome in $\A$. We omit the details, except to say that the arguments in \cite{MR2604856} can be modified for our result simply by ensuring that the block sequences used are taken in $\LH$. This can be done, in each instance, as a block sequence is obtained either by applying the result for open sets or by diagonalization.
\end{proof}

Theorem \ref{thm:local_Rosendal} is consistently sharp and necessarily asymmetric, as there is a coanalytic counterexample (for $\LH=\bb^\infty(E)$) in $\L$ \cite{MR2179777}.\footnote{This counterexample is to Gowers' theorem, but the discussion in \S5 of \cite{MR2604856} shows that this also yields a counterexample to Rosendal's dichotomy.} In particular, the collection of $\LH$-strategically Ramsey sets may fail to be a $\sigma$-algebra. It is, however, closed under countable unions. Again, the proof is nearly identical to that of the corresponding result in \cite{MR2604856} and omitted.

\begin{thm}[cf. Theorem 9 in \cite{MR2604856}]
	Let $\LH\subseteq\bb^\infty(E)$ be a $(p^+)$-family. Then, the collection of $\LH$-strategically Ramsey sets is closed under countable unions.\qed
\end{thm}

We note that fullness is a necessary assumption for our results:

\begin{prop}\label{prop:full_necessary}
	If $\LH\subseteq\bb^\infty(E)$ is a family for which clopen sets are $\LH$-strategically Ramsey, then $\LH$ is full.	
\end{prop}

\begin{proof}
	Given $D\subseteq E$, $\LH$-dense below some $X\in\LH$, let $\D = \{(z_n):z_0\in D\}$, a clopen subset of $\bb^\infty(E)$. For no $Y\in\LH\restr X$ can II have a strategy into $\D^c$: Consider the round of $G[Y]$ where I starts by playing some $Z\preceq Y$ with $\langle Z\rangle\subseteq D$. Since $\D^c$ is $\LH$-strategically Ramsey, there is a $Y\in\LH\restr X$ such that I has a strategy $\sigma$ in $F[Y]$ for playing into $\D$. Let $Z=Y/\sigma(\emptyset)\in\LH$. Since $\sigma$ is a strategy for playing into $\D$, $\langle Z\rangle\subseteq D$.
\end{proof}

\section{Stronger properties of families}\label{sec:4}

If an element $Y$ in a family $\LH$ witnesses Theorem \ref{thm:local_Rosendal}, then either $\A^c$ or $\A$ is $\LH$-dense below $Y$, depending on which half of the dichotomy holds. However, it would be desirable to ensure that $\LH$ itself meets whichever one of $\A^c$ or $\A$ the conclusion of the dichotomy provides. To this end, we consider stronger properties of families, the first of which is based on the original definition of selectivity (or being ``happy'') given in \cite{MR0491197}.

\begin{defn}
	\begin{enumerate}
		\item For $(X_{\vec{x}})_{\vec{x}\in\bb^{<\infty}(E)}$ generating a filter in $\bb^\infty(E)$, we say that $X\in\bb^\infty(E)$ \emph{strongly diagonalizes $(X_{\vec{x}})$} if $X/\vec{x}\preceq X_{\vec{x}}$ whenever $\vec{x}\sqsubseteq X$.
		\item A family $\LH\subseteq\bb^\infty(E)$ is a \emph{strong $(p)$-family}, or has the \emph{strong $(p)$-property}, if whenever $(X_{\vec{x}})_{\vec{x}\in\bb^{<\infty}(E)}$ generates a filter in $\LH$, there is an $Y\in\LH$ which strongly diagonalizes $(X_{\vec{x}})$.
	\end{enumerate}
\end{defn}

The strong $(p)$-property implies the $(p)$-property: Take $X_0\succeq X_1\succeq\cdots$ in $\LH$ and define $X_{\vec{x}}=X_{|\vec{x}|}$ for $\vec{x}\in\bb^{<\infty}(E)$. Any $X$ strongly diagonalizing $(X_{\vec{x}})$ will diagonalize $(X_n)$.

As in Lemma \ref{lem:dense_sets_diag_full}, it is useful for constructing families with the strong $(p)$-property to know that it corresponds to certain $\preceq$-dense sets.

\begin{lemma}\label{lem:dense_sets_strong_diag}
	For $(X_{\vec{x}})_{\vec{x}\in\bb^{<\infty}(E)}$ generating a filter in $\bb^\infty(E)$, the set
	\begin{align*}
		&\{Y:Y\text{ is a strong diagonalization of $(X_{\vec{x}})_{\vec{x}\in\bb^{<\infty}(E)}$, or}\\
		&\hspace{4cm}\{Y\}\cup(X_{\vec{x}})_{\vec{x}\in\bb^{<\infty}(E)} \text{ does not generate a filter}\}
	\end{align*}
	is $\preceq$-dense.
\end{lemma}

\begin{proof}
	Fix $X\in\bb^\infty(E)$, and suppose that $\{X\}\cup(X_{\vec{x}})$ generates a filter. We build a $Y\preceq X$ which strongly diagonalizes $(X_{\vec{x}})$: Pick any $y_0\in\langle X\rangle\cap \langle X_{\emptyset}\rangle$. Since $X$, $X_\emptyset$ and $X_{(y_0)}$ generate a filter, there is a $y_1\in\langle X\rangle\cap\langle X_\emptyset\rangle\cap\langle X_{(y_0)}\rangle$ with $y_0<y_1$. Continue in this fashion.
\end{proof}

The following result connects the strong $(p)$-property to the infinite asymptotic game and is based on a characterization of selective ultrafilters (Theorem 4.5.3 in \cite{MR1350295}).

\begin{thm}\label{thm:strong_p_strat_I}	
	If $\LH\subseteq\bb^\infty(E)$ is a strong $(p)$-family, then for no $X\in\LH$ does I have a strategy in $F[X]$ for playing into $\LH^c$.\footnote{When $\LF$ is a strong $(p)$-filter, one can improve the conclusion to: for no $X\in\LF$ does I have a strategy in $G_\LF[X]$ for playing into $\LF^c$. Here, $G_{\LF}[X]$ is the variant of the Gowers game below $X$ where I is restricted to playing elements of $\LF$. See \S3.11 of \cite{Smythe_thesis}.}
\end{thm}

\begin{proof}
	Let $\sigma$ be a strategy for I in $F[X]$ for playing into $\LH^c$, where $X\in\LH$. Towards a contradiction, suppose that $\LH$ has the strong $(p)$-property. Define sets $\LA_{\vec{x}}\subseteq\LH$ as follows: $\LA_\emptyset=\{X/\sigma(\emptyset)\}$ and inductively, for $\vec{x}=(x_0,\ldots,x_{n-1})$, $\LA_{\vec{x}}$ is the set of all $X/m$ where $m$ is played by I following $\sigma$ in the first $n$ rounds of $F[X]$, as II plays $x_0,x_1,\ldots,x_{n-1}$. In the case that elements of a given $\vec{x}$ fail to be valid moves for II against $\sigma$, let $\LA_{\vec{x}}=\LA_{\vec{x}'}$ where $\vec{x}'$ is the maximal initial segment of $\vec{x}$ consisting of valid moves. Then, for all $\vec{x}$, $\LA_{\vec{x}}$ is finite and $\LA_{\vec{x}}\subseteq\LA_{\vec{y}}$ whenever $\vec{x}\sqsubseteq\vec{y}$.
	
	For each $\vec{x}$, let $M_{\vec{x}}=\max\{m:X/m\in\LA_{\vec{x}}\}$ and $Y_{\vec{x}}=X/M_{\vec{x}}$. Clearly $(Y_{\vec{x}})$ generates a filter in $\LH$. By the strong $(p)$-property, there is a $Y=(y_n)\in\LH\restr X$ such that $Y/\vec{y}\preceq Y_{\vec{y}}$ for all $\vec{y}\sqsubseteq Y$.
	
	Consider the play of $F[X]$ wherein I follows $\sigma$ and II plays $y_0$, $y_1$, and so on. We claim that this is a valid sequence of moves for II. Note that $y_0\in\langle Y/\emptyset\rangle\subseteq\langle Y_\emptyset\rangle\subseteq\langle X/\sigma(\emptyset)\rangle$, so $y_0$ is a valid move. Inductively, suppose that $(y_0,\ldots,y_k)$ is a valid sequence of moves. We have $y_{k+1}\in\langle Y/(y_0,\ldots,y_k)\rangle\subseteq \langle Y_{(y_0,\ldots,y_k)}\rangle\subseteq\langle X/\sigma(y_0,\ldots,y_k)\rangle$, where the last containment uses our induction hypothesis. Thus, $y_{k+1}$ is a valid move. Since the resulting outcome in this play is in $\LH$, we have a contradiction.
\end{proof}

Equivalently, Theorem \ref{thm:strong_p_strat_I} says that if $\LH$ is a strong $(p)$-family and $\sigma$ is a strategy for I in $F[X]$, for $X\in\LH$, then there is an outcome of $\sigma$ in $\LH$.

\begin{lemma}\label{lem:no_strat_into_dense_comp}
	If $\D\subseteq\bb^\infty(E)$ is $\preceq$-dense open below $X\in\bb^\infty(E)$, then
	\begin{enumerate}
		\item II has a strategy in $F[X]$ for playing into $\D$, and
		\item I has a strategy in $G[X]$ for playing into $\D$.	
	\end{enumerate}
\end{lemma}

\begin{proof}
	For $F[X]$, take $Y\preceq X$ in $\D$ and let II always play vectors in $Y$. For $G[X]$, take $Y\preceq X$ in $\D$, and let I simply play $Y$ repeatedly.
\end{proof}

It follows from Lemma \ref{lem:no_strat_into_dense_comp}, and Theorems \ref{thm:local_Rosendal} and \ref{thm:strong_p_strat_I}, that whenever $\LH\subseteq\bb^\infty(E)$ is a strong $(p^+)$-family and $\D$ is a coanalytic $\preceq$-dense open set, then $\LH\cap\D\neq\emptyset$. In particular, strong $(p^+)$-families meet all $\preceq$-dense open Borel sets. This is a special case of Theorem \ref{thm:gen_over_L(R)}. The following definition is a counterpart to Theorem \ref{thm:strong_p_strat_I} for II in $G[X]$.

\begin{defn}
	A family $\LH\subseteq\bb^\infty(E)$ is \emph{strategic} if whenever $X\in\LH$ and $\alpha$ is a strategy for II in $G[X]$, there is an outcome of $\alpha$ which is in $\LH$.
\end{defn}

As as above, if $\LH\subseteq\bb^\infty(E)$ is a strategic $(p^+)$-family and $\D\subseteq\bb^\infty(E)$ is an analytic $\preceq$-dense open set, then $\D\cap\LH\neq\emptyset$. As a consequence for $(p^+)$-filters, being strategic subsumes the strong $(p)$-property.

\begin{prop}
	If $\LF\subseteq\bb^\infty(E)$ is a strategic $(p^+)$-filter, then $\LF$ is also a strong $(p)$-filter.
\end{prop}

\begin{proof}
	Suppose that $\LF$ is as described and $(X_{\vec{x}})_{\vec{x}\in\bb^{<\infty}(E)}$ is contained $\LF$. Let $\LD$ be the set given in Lemma \ref{lem:dense_sets_strong_diag}, so that the $\preceq$-downwards closure of $\LD$ is a $\preceq$-dense open set. Moreover, $\LD$ is easily seen to be Borel and its $\preceq$-downwards closure analytic. By the comments above, it follows that $\LF\cap\LD\neq\emptyset$, and any $Y\in\LF\cap\LD$ must be a strong diagonalization of $(X_{\vec{x}})$.
\end{proof}

In \S\ref{sec:5} we will construct (under set-theoretic hypotheses) strategic $(p^+)$-filters. To this end, we again need to know that certain sets are $\preceq$-dense, but also that there are not ``too many'' of them. If $\alpha$ is a strategy for II in $G[X]$, then the set of outcomes which result from $\alpha$, denoted by $[\alpha,X]$, is $\preceq$-dense below $X$. However, as strategies are functions from finite sequences in $\bb^\infty(E)$ to vectors, there are $2^{2^{\aleph_0}}$ many of them.

One way to resolve this is to ``finitize'' the Gowers game as in \cite{MR1839387}: given $X\in\bb^\infty(E)$, the \emph{finite-dimensional Gowers game below $X$}, denoted by $G_f[X]$, is defined as follows: Two players, I and II, alternate with I going first and playing a non-zero vector $x_0^{(0)}\in\langle X\rangle$. II responds with either a non-zero $y_0\in\langle x_0^{(0)}\rangle$ or $0$. If II plays $y_0$, then the game ``restarts'' with I playing a non-zero vector $x_0^{(1)}\in\langle X\rangle$. If II plays $0$, then I must play a non-zero vector $x_1^{(0)}\in\langle X/x^{(0)}_0\rangle$, to which II again responds with either a non-zero vector $y_0\in\langle x_0^{(0)},x_1^{(0)}\rangle$ or $0$, and so on. The non-zero plays of II are required to satisfy $y_n<y_{n+1}$ and the outcome is the sequence $(y_n)$. The notion of strategy for II in $G_f[X]$ is defined in the obvious way (with the added requirement that the outcome must be infinite) and we denote by $[\alpha,X]_f$ the corresponding set of outcomes.

\begin{lemma}\label{lem:outcomes_strat_dense}
	If $\alpha$ is a strategy for II in $G[X]$, then there is a strategy $\alpha'$ for II in $G_f[X]$ such that $[\alpha',X]_f\subseteq[\alpha,X]$. Moreover, $[\alpha',X]_f$ is still $\preceq$-dense below $X$.
\end{lemma}

\begin{proof}
	The proof is identical to the $(\Leftarrow)$ direction of Theorem 1.2 in \cite{MR1839387}.
\end{proof}

It is easy to see that strategies $\alpha$ for II in $G_f[X]$ are coded by reals and $[\alpha,X]_f$ is an analytic set. This will suffice for our constructions in \S\ref{sec:5}.

\section{Constructions of filters in $\bb^\infty(E)$}\label{sec:5}

In this section we show how to construct filters $\LF\subseteq\bb^{\infty}(E)$ having all of the properties discussed in \S\ref{sec:2} and \S\ref{sec:4}. These constructions use either assumptions about certain ``cardinal invariants'' (cf.~\cite{MR2768685}) which hold consistently with $\ZFC$, or the method of forcing. We will see in Corollary \ref{cor:full_filters_ind} that we cannot hope for a construction in $\ZFC$ alone.

\begin{defn}
	\begin{enumerate}
		\item A \emph{tower} (of \emph{length} $\kappa$) in $\bb^\infty(E)$ is a sequence $(X_\alpha)_{\alpha<\kappa}$ such that $\alpha<\beta<\kappa$ implies $X_\beta\preceq^*X_\alpha$ and there is no $X\in\bb^\infty(E)$ with $X\preceq^* X_\alpha$ for all $\alpha<\kappa$.
		\item $\t^*$ is the minimum length of a tower in $\bb^\infty(E)$.
	\end{enumerate}
\end{defn}

$\t^*$ is a regular cardinal and, moreover, uncountable as $\bb^\infty(E)$ has the $(p)$-property. Thus, the Continuum Hypothesis ($\CH$) implies that $\t^*=2^{\aleph_0}$.

We use the following notational conventions for versions of Martin's Axiom: for $\kappa<2^{\aleph_0}$, $\MA(\kappa)$ is the forcing axiom for meeting $\kappa$-many dense subsets of a ccc poset, $\MA$ is $\forall\kappa<2^{\aleph_0}(\MA(\kappa))$, and $\MA(\sigma\text{-centered})$ is $\MA$ restricted to $\sigma$-centered posets.

\begin{lemma}[Lemma 5 in \cite{MR2145797}]
	\textup{($\MA(\sigma\text{-centered})$)} If $\LL\subseteq\bb^\infty(E)$ is linearly ordered with respect to $\preceq^*$ and $|\LL|<2^{\aleph_0}$, then there is a $Y$ such that $Y\preceq^* X$ for all $X\in\LL$. In particular, $\t^*=2^{\aleph_0}$.
\end{lemma}

Consequently, the following theorem holds under $\CH$ or $\MA(\sigma\text{-centered})$.

\begin{thm}\label{thm:exist_filters}
	$(\t^*=2^{\aleph_0})$ There exists a strategic $(p^+)$-filter in $\bb^\infty(E)$.
\end{thm}

\begin{proof}
	Fix enumerations:
	\begin{enumerate}[label=\rm{(\roman*)}]
		\item $\{X_\xi:\xi<2^{\aleph_0}\}=\bb^\infty(E)$,
		\item $\{(X_n^\xi):\xi<2^{\aleph_0}\}$ of all $\preceq^*$-decreasing sequences $(X^\xi_n)$ in $\bb^\infty(E)$,
		\item $\{D_\xi:\xi<2^{\aleph_0}\}$ of all subsets $D_\xi$ of $E$,
		\item $\{[\alpha_\xi,X_\xi]_f:\xi<2^{\aleph_0}\}$ of all sets $[\alpha,X]_f$ of outcomes of $\alpha$, where $\alpha$ is a strategy for II in $G_f[X]$.
	\end{enumerate}
	This can be done in (i) and (ii) since $|\bb^\infty(E)|=2^{\aleph_0}$, in (iii) since $E$ is countable, and in (iv) since the strategies $\alpha$ are coded by reals.
		
	Define sets, for $\xi,\gamma<2^{\aleph_0}$, with $\langle\cdot,\cdot\rangle$ a bijection $2^{\aleph_0}\times 2^{\aleph_0}\to 2^{\aleph_0}$,
	\begin{align*}
		\LD_\xi &=\{Y:Y\text{ is a diagonalization of } (X^\xi_n)\text{ or }\exists n(Y\bot X^\xi_n)\}\\
		\LF_{\langle\xi,\gamma\rangle}&=\{Y:\langle Y\rangle\subseteq D_\xi \text{ or } \forall V\preceq X_\gamma(\langle V\rangle\subseteq D_\xi \Rightarrow V\perp Y)\}\\
		\LS_{\xi}&=\{Y:Y\in[\alpha_\xi,X_\xi]_f \text{ or } Y\perp X_\xi\}.
	\end{align*}
	Note that the first two sets above are $\preceq$-dense in $\bb^\infty(E)$ by Lemma \ref{lem:dense_sets_diag_full}, and the third is $\preceq$-dense by Lemma \ref{lem:outcomes_strat_dense}.
	
	We construct a $\preceq^*$-descending chain $(Y_\eta)$ of length $2^{\aleph_0}$ in $\bb^\infty(E)$ by transfinite induction on $\eta$. For $\eta=0$, pick $Y_0$ below conditions in each of $\LD_0$, $\LF_0$, and $\LS_0$. If we have already defined $Y_\beta$ for all $\beta<\eta$, pick $Y_\eta$ below each $Y_\beta$ for $\beta<\eta$ and below conditions in each of $\LD_\eta$, $\LF_\eta$, and $\LS_\eta$. This is possible since $t^*>\eta$. 
	
	Let $\LF$ be the filter generated by $\{Y_\eta:\eta<2^{\aleph_0}\}$ in $\bb^\infty(E)$. To see that $\LF$ is a $(p)$-filter, suppose that $(X_n^\xi)$ is a $\preceq^*$-decreasing sequence in $\LF$. Let $Y\in\LF\cap\LD_\xi$. It cannot be the case that $Y\bot X_n^\xi$ for any $n$, as $\LF$ is a filter, so $Y$ must be a diagonalization of $(X^\xi_n)$. Similarly, using the sets $\LS_\xi$, $\LF$ is strategic.
	
	To see that $\LF$ is full, suppose $D_\xi\subseteq E$ and $X_\gamma\in\LF$ are such that $D_\xi$ is $\LF$-dense below $X_\gamma$. Take $Z\in\LF\cap\LF_{\langle\xi,\gamma\rangle}\neq\emptyset$. By assumption, there is a $Y'$ below both $Y$ and $X_\gamma$ such that $\langle Y'\rangle\subseteq D_\xi$, but obviously it cannot be that $Y'\perp Y$. Thus, it must be that $\langle Y\rangle\subseteq D_\xi$.
\end{proof}

The next result allows us to obtain $(p^+)$-filters generically by forcing with $(\bb^\infty(E),\preceq^*)$. Since the dense sets involved are all definable in a simple way from real parameters, they are contained in $\L(\R)$. In particular, this establishes (without any large cardinals) the $(\Rightarrow)$ direction of Theorem \ref{thm:gen_over_L(R)}.

\begin{lemma}\label{lem:forcing_family}
	For $\LH\subseteq\bb^\infty(E)$ a $(p^+)$-family, forcing with $(\LH,\preceq^*)$ adds no new reals and if $\LG\subseteq\LH$ is $\L(\R)$-generic for $(\LH,\preceq^*)$, then $\LG$ will be a $(p^+)$-filter. If $\LH$ is strategic (has the strong $(p)$-property, respectively), then $\LG$ will also be strategic (have the strong $(p)$-property, respectively).
\end{lemma}

\begin{proof}
	$\LH$ being a $(p)$-family implies that $(\LH,\preceq^*)$ is $\sigma$-closed, and thus adds no new reals. We use this fact implicitly in what follows. Let $\LG$ be as described. To see that $\LG$ is full, let $D\subseteq E$ be $\LG$-dense below some $X\in\LG$. Translating this into the forcing language, there must be an $X'\in\LG$, which we may assume is below $X$, with
	\[
		X'\forces_{\LH}\forall Y\in\dot{\LG}\restr X\exists Z\preceq Y(\langle Z\rangle\subseteq\check{D}).
	\]
	We claim that the set $\LD=\{Z:\langle Z\rangle\subseteq D\}$ is $\preceq$-dense below $X'$ in $\LH$. If not, then by fullness of $\LH$, $D$ must fail to be $\LH$-dense below $X'$. That is, there is some $Y\in\LH\restr X'$ with no $Z\preceq Y$ such that $\langle Z\rangle \subseteq D$. Then, $Y$ fails to force the statement in the displayed line above, contrary to $Y\preceq X'$. Since $X'\in\LG$ and $\LD$ is $\preceq$-dense below $X'$ in $\LH$, $\LG\cap\LD\neq\emptyset$, showing that $\LG$ is full. The remainder of the proof consists of observing that the relevant $\preceq$-dense sets in Lemmas \ref{lem:dense_sets_diag_full}, \ref{lem:dense_sets_strong_diag}, and \ref{lem:outcomes_strat_dense} are $\preceq$-dense in $\LH$ under these hypotheses.
\end{proof}

\section{Connections to filters on a countable set}\label{sec:6}

We would like to relate the filters discussed thus far to filters of subsets of a countable set. In our case, the countable set will be $E\setminus\{0\}$, but we will call these filters on $E$.

\begin{defn}
	A filter $\LF$ on $E$ is a \emph{block filter} if it has a base consisting of sets of the form $\langle X\rangle$, for $X\in\bb^\infty(E)$.
\end{defn}

It is tempting to define a \emph{block ultrafilter} on $E$ to be a block filter on $E$ which is also an ultrafilter. However, unless $|F|=2$, such objects do not exist: Let $\LF$ be a block filter on $E$. For $A_0,A_1\subseteq E$ given in Example \ref{ex:asym_sets}, note that $E=A_0\cup A_0$. But, for every $X\in\bb^\infty(E)$, $\langle X\rangle\cap A_0\neq\emptyset$ and $\langle X\rangle\cap A_1\neq\emptyset$, so neither set can be in $\LF$.

Let $\FIN$ be the set of nonempty finite subsets of $\N$. An ultrafilter $\LU$ on $\FIN$ is said to be an \emph{ordered union ultrafilter} \cite{MR906807} if it has a base consisting of sets of the form $\langle X\rangle=\{x_{n_0}\cup\cdots\cup x_{n_k}:n_0<\cdots<n_k\}$, where $X=(x_n)$ is a block sequence in $\FIN$ (that is, for all $n$, $\max(x_n)<\min(x_{n+1})$). The set of infinite block sequences in $\FIN$ is denoted by $\FIN^{[\infty]}$. We have, perhaps, overloaded the notation $\langle X\rangle$, but its intended interpretation should be clear from context. If $X=(x_n)\in\bb^\infty(E)$, denote by $\tilde{X}=(\supp(x_n))\in\FIN^{[\infty]}$.

If $|F|=2$, then $E\setminus\{0\}$ can be identified with $\FIN$ via each vector's support. Sums of vectors in block position corresponds to unions of their supports. As a consequence of Hindman's Theorem (Corollary 3.3 in \cite{MR0349574}), one can construct (under hypotheses such as $\CH$ or $\MA$) ordered union ultrafilters on $\FIN$; these will correspond to block ultrafilters on $E$. 

For the remainder of this section we will consider a general countable field $F$. The map which takes a vector to its support will provide the connection between this general setting and $\FIN$.

\begin{defn}
	Let $\LF$ be a block filter on $E$.
	\begin{enumerate}
		\item A subset $D\subseteq E$ is \emph{$\LF$-dense} if for every $\langle X\rangle\in\LF$, there is a $Z\preceq X$ with $\langle Z\rangle\subseteq D$.
		\item $\LF$ is \emph{full} if whenever $D\subseteq E$ is $\LF$-dense, we have that $D\in\LF$.
	\end{enumerate}
\end{defn}

As in the case for filters in $\bb^\infty(E)$, every full block filter on $E$ is maximal with respect to containment amongst block filters.

The map $s:X\mapsto\langle X\rangle$ takes block sequences to subsets of $E$. It is straightforward to show that the image of a (full) filter in $\bb^\infty(E)$ under $s$ generates a (full) block filter on $E$ and that the inverse image of a (full) block filter on $E$ is a (full) filter in $\bb^\infty(E)$. By Theorem \ref{thm:exist_filters} (or Lemma \ref{lem:forcing_family}), it is consistent that such filters exist.

\begin{thm}\label{thm:full_block_ordered_union}
	Suppose that $\LF$ is a full block filter on $E$, and let
	\[
		\supp(\LF)=\{A\subseteq\FIN:\exists F\in\LF(A\supseteq\{\supp(v):v\in F\})\}.
	\]
	Then, $\supp(\LF)$ is an ordered union ultrafilter on $\FIN$.
\end{thm}

\begin{proof}
	Let $A,B\in\supp(\LF)$, say with $A\supseteq\{\supp(v):v\in F\}$ and $B\supseteq\{\supp(v):v\in G\}$, for $F,G\in\LF$. Then,
	\begin{align*}
		A\cap B&\supseteq\{s:\exists v\in F\exists w\in G(s=\supp(v)=\supp(w))\}\\
		&\supseteq\{\supp(v):v\in F\cap G\},
	\end{align*}
	which is in $\supp(\LF)$, as $F\cap G\in\LF$. Since $\supp(\LF)$ is upwards closed by definition, we have that $\supp(\LF)$ is a filter on $\FIN$. As $\LF$ is a block filter, it follows that $\supp(\LF)$ has a base consisting of sets $\langle\tilde{X}\rangle$ for $X\in\bb^\infty(E)$.

	It remains to show that $\supp(\LF)$ is an ultrafilter. Take $A\subseteq\FIN$ such that for all $B\in\supp(\LF)$, $A\cap B\neq\emptyset$. Let
	\begin{align*}
		D_0 &= \{v\in E:\supp(v)\in A\}\\
		D_1 &= \{v\in E:\supp(v)\notin A\}.	
	\end{align*}
	Towards a contradiction, suppose that for all $\langle X\rangle\in\LF$, there is a $\langle Z\rangle\subseteq \langle X\rangle$ with $\langle Z\rangle\subseteq D_1$. Since $\LF$ is full, there is a $\langle Z\rangle\in\LF$ with $\langle Z\rangle\subseteq D_1$. Then $\langle\tilde{Z}\rangle\in\supp(\LF)$, but $A\cap\langle\tilde{Z}\rangle=\emptyset$, a contradiction.
	
	Thus, there is some $\langle X\rangle\in\LF$ such that for no $\langle Z\rangle\subseteq\langle X\rangle$ is $\langle Z\rangle\subseteq D_1$. Take $\langle Y\rangle\in\LF\restr\langle X\rangle$. By Hindman's Theorem applied to $\langle\tilde{Y}\rangle$, there is a $\tilde{Z}\in\FIN^{[\infty]}$ such that $\langle \tilde{Z}\rangle\subseteq\langle \tilde{Y}\rangle$ and either (i) $\langle \tilde{Z}\rangle\subseteq A$, or (ii) $\langle \tilde{Z}\rangle\subseteq \langle\tilde{Y}\rangle\setminus A$.
	
	Take any $Z\preceq Y$ in $\bb^\infty(E)$ whose supports agree with $\tilde{Z}$, then if (ii) holds, $\langle Z\rangle\subseteq D_1$, contrary to what we know about $\langle X\rangle$. Thus, $\langle\tilde{Z}\rangle\subseteq A$ and $\langle Z\rangle\subseteq D_0$. Since $\langle Y\rangle\in\LF\restr\langle X\rangle$ was arbitrary, we have that $D_0$ is $\LF$-dense. As $\LF$ is full, we can find a $\langle Z\rangle\in\LF$ with $\langle Z\rangle\subseteq D_0$. Then, $\langle\tilde{Z}\rangle\in\supp(\LF)$ and $\langle\tilde{Z}\rangle\subseteq A$, so $A\in\supp(\LF)$.
\end{proof}

As a consequence of Theorem \ref{thm:full_block_ordered_union} and the Corollary on p.~87 of \cite{MR906807} we have:

\begin{cor}
	If $\LF$ is a full filter on $E$, then
	\begin{align*}
		\min(\LF)&=\{\{n=\min(\supp(v)):v\in F\}:F\in\LF\},\\
		\max(\LF)&=\{\{n=\max(\supp(v)):v\in F\}:F\in\LF\}
	\end{align*}
	are selective ultrafilters on $\N$.
\end{cor}

As it is consistent that there are no selective ultrafilters \cite{MR0427070}, we have:

\begin{cor}\label{cor:full_filters_ind}
	The existence of full block filters on $E$, and thus full filters in $\bb^\infty(E)$, is independent of $\ZFC$.
\end{cor}

An ordered union ultrafilter $\LU$ on $\FIN$ is \emph{stable} \cite{MR891244} if whenever $(\langle X_n\rangle)_{n\in\N}$ is contained in $\LU$, for $X_n\in\FIN^{[\infty]}$, there is an $\langle X\rangle\in\LU$ with $\langle X\rangle\subseteq^*\langle X_n\rangle$ for all $n$. Much as selective ultrafilters on $\N$ provide local witnesses to Silver's theorem, selective ultrafilters on $\FIN$ witness a theorem of Milliken \cite{MR0373906} on analytic partitions of $\FIN^{[\infty]}$. It is easy to see, given Theorem \ref{thm:full_block_ordered_union}, that $(p^+)$-filters in $\bb^\infty(E)$ induce stable ordered union ultrafilters on $\FIN$. See \cite{local_Ramsey}, \cite{MR2330595}, and \cite{yyz_sel_fin} for (equivalent) alternate definitions of ``selective ultrafilter'' on $\FIN$.

\section{Extending to universally Baire sets and $\L(\R)$}\label{sec:7}

In this section, we show that under additional set-theoretic hypotheses, Theorem \ref{thm:local_Rosendal} can be extended beyond the analytic sets to obtain Theorems \ref{thm:gen_over_L(R)} and \ref{thm:local_Rosendal_L(R)}, provided the families involved are strategic. We begin by noting the following result:

\begin{thm}[Rosendal \cite{MR2604856}]\label{thm:aleph_1_strat_Ramsey}
	\textup{($\MA(\aleph_1)$)} A union of $\aleph_1$-many strategically Ramsey sets is strategically Ramsey.
\end{thm}

The above theorem, plus existing results in the literature, yields:

\begin{thm}\label{thm:L(R)_strat_Ramsey}
	Assume that there is a supercompact cardinal.\footnote{Throughout this section, the assumption of supercompactness can be weakened to the existence of a proper class of Woodin cardinals, see \cite{MR2069032}. We use supercompactness due to its central role in the literature and verbal brevity.} Every subset of $\bb^\infty(E)$ in $\L(\R)$ is strategically Ramsey.\footnote{No\'{e} de Rancourt has announced a different proof of this result using methods inspired by determinacy considerations.}
\end{thm}

\begin{proof}
	We follow the proof of Theorem 4 in \cite{MR2179777}. The existence of a supercompact cardinal implies that $\L(\R)$ is a \emph{Solovay model} in the sense of \cite{MR1978344}, and Lemma 4.4 of the same reference shows that every set of reals in such a model is a union of $\aleph_1$-many analytic sets. By Theorem \ref{thm:aleph_1_strat_Ramsey}, under $\MA(\aleph_1)$ a union of $\aleph_1$-many strategically Ramsey sets is again strategically Ramsey. Since supercompactness implies \cite{MR1074499} that $\L(\R)^{\V[G]}$ is elementarily equivalent to $\L(\R)$ for any set-forcing extension $\V[G]$, and one can force $\MA(\aleph_1)$ in a way which preserves $\aleph_1$, the same is true in $\L(\R)$. As analytic sets are strategically Ramsey by Theorem \ref{thm:local_Rosendal}, every set in $\L(\R)$ is as well.
\end{proof}

Following \cite{MoForcing}, given a notion of forcing $Q$ and a complete metric space $(X,d)$, we say that a $Q$-name $\dot{x}$ is a \emph{nice $Q$-name for an element of $\dot{X}$} if there is a countable collection $\LD$ of dense subsets of $Q$ such that $\dot{x}(G)$ (the interpretation of $\dot{x}$ by $G$) is an element of $X$ whenever $G$ is a $\LD$-generic filter for $Q$. One can show that if $\dot{y}$ is a $Q$-name and $p\forces_Q\dot{y}\in\dot{X}$, then there is a nice $Q$-name $\dot{x}$ for an element of $\dot{X}$ such that $p\forces_Q\dot{y}=\dot{x}$.

A subset $A\subseteq X$ is \emph{universally Baire} if whenever $Q$ is a notion of forcing, there is a $Q$-name $\dot{A}$ such that for every nice $Q$-name $\dot{x}$ for an element of $\dot{X}$, there is a countable collection $\LD$ of dense subsets of $Q$ such that
\begin{enumerate}[label=\rm{(\arabic*)}]
	\item $\{q\in Q:q \text{ decides } \dot{x}\in\dot{A}\}$ is in $\LD$,
	\item whenever $G$ is $\LD$-generic for $Q$, $\dot{x}(G)$ is in $X$ and $\dot{x}(G)$ is in $A$ if and only if there is a $q\in G$ such that $q\forces\dot{x}\in\dot{A}$.	
\end{enumerate}

The following result will be the main tool for going beyond the analytic sets.

\begin{thm}[Feng--Magidor--Woodin \cite{MR1233821}]\label{thm:FMW_UBsets}
	Assume that there is a supercompact cardinal. Every set of reals in $\L(\R)$ is universally Baire.
\end{thm}

Consider the following variant of the infinite asymptotic game: If $A\subseteq E$ is an infinite dimensional subspace of $E$, we define $F[A]$ to be the game in which I plays natural numbers $n_k$, which we assume are increasing, and II plays vectors $y_k\in A$ subject to the constraint $n_k<y_k<y_{k+1}$. By Lemma \ref{lem:subspaces_cont_blocks}, this is well-defined. One can define \emph{outcome}, \emph{strategies}, and the game $F[\vec{x},A]$ exactly as in \S\ref{sec:3}. Note that the game $F[\vec{x},\langle X\rangle]$ in this sense, where $X\in\bb^\infty(E)$, coincides with $F[\vec{x},X]$ from \S\ref{sec:3}, and we will denote it as such.

Suppose that $\sigma$ is a strategy for I in $F[A]$ and $\tau$ a strategy for I in $F[B]$, where $B\subseteq A$ are infinite-dimensional subspaces. We write $\tau\geq\sigma$ if for all $\vec{y}$ in the domain of $\tau$, $\tau(\vec{y})\geq\sigma(\vec{y})$ ($\sigma(\vec{y})$ is well-defined by induction). Observe that if $\tau\geq\sigma$, then whenever $(y_n)$ is an outcome of $F[B]$ where I follows $\tau$, then it is also an outcome of $F[A]$ where I follows $\sigma$. In particular, if $\sigma$ is a strategy for playing into a set $\A$, then so is $\tau$.

If $\sigma$ is a strategy for I in $F[A]$ and $B\subseteq A$ as above, then denote by $\sigma\restr B$ the restriction of $\sigma$ to the part of its domain contained in $B$, a strategy for I in $F[B]$. Clearly, $\sigma\restr B\geq \sigma$. Let $\varepsilon$ be the strategy in $F[E]$ where I plays $n$ on the $n$th move. Then, for all $A$ and strategies $\sigma$ for $I$ in $F[A]$, we have that $\sigma\geq\varepsilon\restr A$.

\begin{defn}
	Let $\P$ be the set of all triples $(\vec{x},A,\sigma)$, where $\vec{x}\in\bb^{<\infty}(E)$, $A$ is an infinite-dimensional subspace of $E$ and $\sigma$ is a strategy for I in $F[\vec{x},A]$. We say that $(\vec{y},B,\tau)\leq(\vec{x},A,\sigma)$ if
	\begin{enumerate}[label=\rm{(\roman*)}]
		\item $\vec{y}=\vec{x}\concat(y_0,\ldots,y_{k-1})$ where $y_0,\ldots,y_{k-1}$ are the first $k$ moves by II in a round of $F[\vec{x},A]$ where I follows $\sigma$,
		\item $B\subseteq A$,
		\item $\tau(\cdot)\geq\sigma((y_0,\ldots,y_k)\concat\,\cdot\,)$.
	\end{enumerate}		
\end{defn}

The ordering $\leq$ on $\P$ is reflexive and transitive, though fails to be antisymmetric. We treat $\P$ as a notion of forcing. Note that $\P$ has a maximal element, namely $(\emptyset,E,\varepsilon)$. If $X\in\bb^\infty(E)$, we write $(\vec{x},X,\sigma)$ for $(\vec{x},\langle X\rangle,\sigma)$. If $\LH\subseteq\bb^\infty(E)$ is a family, let 
\[
	\P(\LH)=\{(\vec{x},A,\sigma)\in\P:\exists X\in\LH(\langle X\rangle\subseteq A)\},
\]
a suborder of $\P$. Note that if $\LH\subseteq\bb^\infty(E)$ is a family, then the set of conditions $(\vec{x},X,\sigma)$ where $X\in\LH$ is dense in $\P(\LH)$.

For $(\vec{x},A,\sigma)\in\P$, let
\[
[\vec{x},A,\sigma]=\{Y\in\bb^\infty(E):Y\text{ is an outcome of $F[\vec{x},A]$ where I follows $\sigma$}\}.
\]

We collect some basic properties of $\P$ in the following lemma:

\begin{lemma}\label{lem:P_basic_props}
	\begin{enumerate}
		\item If $(\vec{y},B,\tau)\leq(\vec{x},A,\sigma)$ in $\P$, then $[\vec{y},B,\tau]\subseteq[\vec{x},A,\sigma]$. Conversely, if $[\vec{y},B,\tau]\subseteq[\vec{x},A,\sigma]$, then $(\vec{y},B,\tau)$ is below $(\vec{x},A,\sigma)$ in the separative quotient of $\P$.
		\item If $(\vec{x},A,\sigma)\in\P$, then the set $[\vec{x},A,\sigma]$ is (topologically) closed.
		\item If $\LF\subseteq\bb^\infty(E)$ is a filter, then $\P(\LF)$ is $\sigma$-centered.	
	\end{enumerate}
\end{lemma}

\begin{proof}
	(a) The first part follows from our observations about the ordering on strategies for I. For the converse, suppose that $[\vec{y},B,\tau]\subseteq[\vec{x},A,\sigma]$. Then, every outcome of $F[\vec{y},B]$ where I follows $\tau$ is an outcome of $F[\vec{x},A]$ where I follows $\sigma$. In particular, $\vec{y}=\vec{x}\concat(y_0,\ldots,y_{k-1})$ where $y_0,\ldots,y_{k-1}$ are the first $k$ moves by II in a round of $F[\vec{x},A]$ where I follows $\sigma$. 
	
	We claim that $B/m\subseteq A$, where $m=\max\{\supp(\vec{y}),\tau(\emptyset)\}$ and $B/m=\{y\in B:y>m\}$. To see this, note that for any $y\in B/m$, there is an outcome $\vec{y}\concat y\concat Z\in[\vec{y},B,\tau]$ and thus in $[\vec{x},A,\sigma]$. In particular, $y\in A$. 
	
	By our choice of $m$, $\tau\restr B/m=\tau$. So, $(\vec{y},B/m,\tau)\leq(\vec{x},A,\sigma)$ and the sets of extensions of $(\vec{y},B/m,\tau)$ and $(\vec{y},B,\tau)$ coincide. Thus, their images in the separative quotient of $\P$ coincide.
	
	\noindent(b) If $Y=(y_n)\notin[\vec{x},A,\sigma]$, then either $\vec{x}\not\sqsubseteq Y$, or there is some least $n$ such that $y_n$ is not a valid response to $\sigma(y_0,\ldots,y_{n-1})$, i.e., $y_n\notin A$ or $y_n\not>\sigma(y_0,\ldots,y_{n-1})$. As $E$ is discrete, these are open conditions.
		
	\noindent(c) Suppose that $(\vec{x},A,\sigma)$ and $(\vec{x},B,\tau)$ are both in $\P(\LF)$. There are $X,Y\in\LF$ with $\langle X\rangle\subseteq A$ and $\langle Y\rangle\subseteq B$. Since $\LF$ is a filter, there is a $Z\in\LF$ below both. Let $\rho$ be the strategy for I in $F[Z]$ given by $\rho(\vec{z})=\max\{\sigma(\vec{z}),\tau(\vec{z})\}$. Then, $(\vec{x},Z,\rho)\in\P(\LF)$ and extends both $(\vec{x},A,\sigma)$ and $(\vec{x},B,\tau)$. Since there are only countably many such $\vec{x}$, this shows that $\P(\LF)$ is $\sigma$-centered.
\end{proof}

Given a family $\LH\subseteq\bb^\infty(E)$ and a sufficiently generic filter $G$ for $\P(\LH)$, we denote by $X_\text{gen}(G)$ the \emph{generic block sequence} determined by $G$,
\begin{align*}
	X_{\text{gen}}(G)&=\bigcup\{\vec{x}:\exists (\vec{x},A,\sigma)\in G\}.
\end{align*}
In what follows, $G$ will be $\LD$-generic for some countable collection of dense sets $\LD$ coming from the definition of universally Baire, and so $G$ can be taken to be in $\V$. Any such $\LD$ will ensure that $X_{\text{gen}}(G)$ is infinite. We write $\dot{X}_{\text{gen}}$ to be a nice (as defined above) $\P(\LH)$-name for this block sequence.

\begin{lemma}\label{lem:gen_block_P}
	Let $\LF\subseteq\bb^\infty(E)$ be a filter, $\LD$ a collection of dense subsets of $\P(\LF)$, and $G$ a $\LD$-generic filter for $\P(\LF)$. For $X=X_\text{gen}(G)$, the set 
	\[
		G(X)=\{(\vec{x},A,\sigma)\in\P(\LF):X\in[\vec{x},A,\sigma]\}
	\]
	is a $\LD$-generic filter for $\P(\LF)$ which contains $G$ and $X_{\text{gen}}(G(X))=X$.
\end{lemma}

\begin{proof}
	By Lemma \ref{lem:P_basic_props}(a), $G(X)$ is closed upwards. If $(\vec{x},A,\sigma)\in G$, then one can build a decreasing sequence $(\vec{x}_n,A_n,\sigma_n)$ in $G$ with $(\vec{x}_0,A_0,\sigma_0)=(\vec{x},A,\sigma)$, $|\vec{x}_n|\to\infty$ as $n\to\infty$, and $X$ the union of the $\vec{x}_n$. By construction, $X$ must be in $[\vec{x},A,\sigma]$. This shows that $G\subseteq G(X)$, and consequently the latter is $\LD$-generic. It remains to show that $G(X)$ is a filter. Take $(\vec{x},A,\sigma),(\vec{y},B,\tau)\in G(X)$. As $X$ has both $\vec{x}$ and $\vec{y}$ as an initial segment, one must be an initial segment of the other, say $\vec{x}\sqsubseteq\vec{y}$, and the part of $\vec{y}$ above $\vec{x}$ is a sequence of moves by II against $\sigma$. As $\LF$ is a filter, $A\cap B$ is infinite-dimensional. Let $\rho$ be the strategy for I in $F[A\cap B]$ given by $\rho(\vec{v})=\max\{\sigma(\vec{v}),\tau(\vec{v})\}$, for $\vec{v}$ in its domain. Then, $(\vec{y},A\cap B,\rho)$ is below both $(\vec{x},A,\sigma)$ and $(\vec{y},B,\tau)$. Moreover, $X\in[\vec{y},A\cap B,\rho]$, and so $(\vec{y},A\cap B,\rho)\in G(X)$. That $X_{\text{gen}}(G(X))=X$ is clear.
\end{proof}

A consequence of Lemma \ref{lem:gen_block_P} is that if $G$ is generic for $\P(\LF)$ over a model of a sufficient fragment of $\ZFC$, then $G(X)=G$, though we will not make use of this here.

\begin{lemma}\label{lem:gen_blocks_Borel_strat}
	Let $\LF\subseteq\bb^\infty(E)$ be a filter and $\LD$ a countable collection of dense open subsets of $\P(\LF)$.
	\begin{enumerate}
		\item For any $(\vec{x},A,\sigma)\in\P(\LF)$, the set
		\[
			\G_{\LD,(\vec{x},A,\sigma)}=\{X_{\text{gen}}(G):G\text{ a $\LD$-generic filter for $\P(\LF)$ with $(\vec{x},A,\sigma)\in G$}\}
		\]
		is an $F_{\sigma\delta}$ subset of $\bb^\infty(E)$.
		\item If $X\in\LF$, then for no $Y\in\LF\restr X$ does I have a strategy in $F[\vec{x},Y]$ for playing into $(\G_{\LD,(\vec{x},X,\sigma)})^c$.
		\item If $\LF$ is a $(p^+)$-filter and $X\in\LF$, then there is a $Y\in\LF\restr X$ for which II has a strategy in $G[\vec{x},Y]$ for playing into $\G_{\LD,(\vec{x},X,\sigma)}$.
	\end{enumerate}
\end{lemma}

\begin{proof}
	(a) Enumerate $\LD=\{D_n:n\in\N\}$. Since $\P(\LF)$ is ccc by Lemma \ref{lem:P_basic_props}(c), each $D_n$ contains a countable maximal antichain $A_n$ below $(\vec{x},A,\sigma)$. We claim that
	\[
		\G_{\LD,(\vec{x},A,\sigma)}= \bigcap_{n\in\N}\bigcup\{[\vec{y},B,\tau]:(\vec{y},B,\tau)\in A_n\},
	\]
	which is $F_{\sigma\delta}$, as each set $[\vec{y},Y,\tau]$ is closed by Lemma \ref{lem:P_basic_props}(b). 
	
	If $X=X_{\text{gen}}(G)$ where $G$ is a $\LD$-generic filter with $(\vec{x},A,\sigma)\in G$, then for each $n$, $G\cap A_n\neq\emptyset$, say with $(\vec{y}_n,B_n,\tau_n)\in G\cap A_n$. By Lemma \ref{lem:gen_block_P}, for each $n$, $X\in[\vec{y}_n,B_n,\tau_n]$, and so $X$ is in the set on right hand side of the above displayed line. For the reverse inclusion, suppose that $X$ is in set on the right hand side. Then, by Lemma \ref{lem:gen_block_P}, $G(X)$ is a $\LD$-generic filter containing $(\vec{x},A,\sigma)$ for which $X_{\text{gen}}(G(X))=X$, and so $X\in\G_{\LD,(\vec{x},A,\sigma)}$.
	
	(b) Let $X\in\LF$ and $Y\in\LF\restr X$ be given. Towards a contradiction, suppose that $\rho$ is a strategy for I in $F[\vec{x},Y]$ for playing into $(\G_{\LD,(\vec{x},X,\sigma)})^c$. We may assume $\rho\geq\sigma\restr Y$. Consider the following play of $F[\vec{x},Y]$: I plays $\rho(\emptyset)=n_0$. Pick
	\[
		p_0=(\vec{x}\concat(y_0^0,\ldots,y_{k_0}^0),B^0,\rho^0)\leq(\vec{x},Y,\rho)\leq(\vec{x},X,\sigma)
	\]
	in $D_0$ and let II play $y_0=y^0_0$. Note that this is a valid move by definition of $\leq$ in $\P(\LF)$. Next, I plays $\rho(y_0)=n_1$. Pick
	\[
		p_1=(\vec{x}\concat(y_0^0,\ldots,y_{k_0}^0)\concat(y_0^1,\ldots,y_{k_1}^1),B^1,\rho^1)\leq(\vec{x}\concat(y_0^0,\ldots,y_{k_0}^0),B^0,\rho^0)
	\]
	in $D_1$ and let II play $y_1=y^0_1$ if $k_0\geq 1$, and $y_1=y^1_0$ otherwise. Continuing in this fashion, we build an outcome $(y_n)$. Observe that $(y_n)$ must be in $\G_{\LD,(\vec{x},X,\sigma)}$: the conditions $p_n$ picked in $D_n$ above form a $\LD$-generic chain in $\P(\LF)$ below $(\vec{x},X,\sigma)$, thus generate a $\LD$-generic filter $G$ with $X_{\text{gen}}(G)=(y_n)$ and $(\vec{x},X,\sigma)\in G$. This contradicts our choice of $\rho$.
	
	\noindent(c) follows from (a) and (b) by an application of Theorem \ref{thm:local_Rosendal}.
\end{proof}

\begin{lemma}\label{lem:uB_strat}
	Let $\LF\subseteq\bb^\infty(E)$ be a $(p^+)$-filter. If $\A\subseteq\bb^\infty(E)$ is universally Baire, then for any $\vec{x}\in\bb^{<\infty}(E)$ and $X\in\LF$, there is a $Y\in\LF\restr X$ such that II has a strategy in $G[\vec{x},Y]$ for playing into one of $\A$ or $\A^c$.
\end{lemma}

\begin{proof}
	Let $X\in\LF$ be given. We may assume that $\vec{x}$ is $\emptyset$. Recall, for $\vec{y}\in\bb^{<\infty}(E)$ and $Y\in\LF$, Definition \ref{def:good_bad_worse} of $(\vec{y},Y)$ being good/bad/worse (for the set $\A$). By Lemma \ref{lem:good_worse}, there is a $Y\in\LF\restr X$ such that either $(\emptyset,Y)$ is good or I has a strategy $\sigma$ in $F[Y]$ to play into the set
	\[
		\{(z_n):\forall n(z_0,\ldots,z_n,Y)\text{ is worse}\}.
	\]
	In the former case we're done, so we assume the latter.
	
	Since $\A$ is universally Baire, we may let $\dot{\A}$ be a $\P(\LF)$-name for $\A$ and $\LD$ countable collection of dense open subsets of $\P(\LF)$ such that 
	\begin{enumerate}[label=\rm{(\roman*)}]
		\item $\{q\in \P(\LF):q \text{ decides } \dot{X}_{\text{gen}}\in\dot{\A}\}$ is in $\LD$, and
		\item whenever $G$ is $\LD$-generic in $\P(\LF)$, $X_{\text{gen}}(G)$ is in $\bb^\infty(E)$ and $X_{\text{gen}}(G)$ is in $\A$ if and only if there is a $q\in G$ such that $q\forces_{\P(\LF)}\dot{X}_{\text{gen}}\in\dot{\A}$.	
	\end{enumerate}
	Thus, if $G$ is $\LD$-generic for $\P(\LF)$, contains $(\emptyset,Y,\sigma)$, and $(\emptyset,Y,\sigma)\forces_{\P(\LF)}\dot{X}_{\text{gen}}\notin\dot{\A}$, then $X_{\text{gen}}(G)\notin\A$. We claim that $(\emptyset,Y,\sigma)\forces_{\P(\LF)}\dot{X}_{\text{gen}}\notin\dot{\A}$. 
	
	Suppose not, then there is a $(\vec{y},Z,\tau)\leq(\emptyset,Y,\sigma)$, with $Z\in\LF$, such that $(\vec{y},Z,\tau)\forces_{\P(\LF)}\dot{X}_{\text{gen}}\in\dot{\A}$. Applying Lemma \ref{lem:gen_blocks_Borel_strat}(c), take $W\in\LF\restr Z$ such that II has a strategy $\alpha$ in $G[\vec{y},W]$ for playing into $\G_{\LD,(\vec{y},Z,\tau)}$. We claim that $\G_{\LD,(\vec{y},Z,\tau)}\subseteq\A$. Let $(z_n)$ be in $\G_{\LD,(\vec{y},Z,\tau)}$. Take $G$ a $\LD$-generic filter for which $(z_n)=X_{\text{gen}}(G)$ and $(\vec{y},Z,\tau)\in G$. Since $(\vec{y},Z,\tau)\forces_{\P(\LF)}\dot{X}_{\text{gen}}\in\dot{\A}$, we have that $(z_n)\in\A$. Thus, $\alpha$ is a strategy for II in $G[\vec{y},W]$ for playing into $\A$. This, however, contradicts the fact that $\sigma$ ensures that $(\vec{y},Z)$ is bad.
	
	Thus, $(\emptyset,Y,\sigma)\forces_{\P(\LF)}\dot{X}_{\text{gen}}\notin\dot{\A}$. Then, exactly as in the preceding paragraph, we may find $W\in\LF\restr Y$ such that II has a strategy in $G[W]$ for playing into $\G_{\LD,(\emptyset,Y,\sigma)}$, and thus into $\A^c$.
\end{proof}

While the symmetric result in Lemma \ref{lem:uB_strat} is appealing on its own, and applies to all analytic sets (being universally Baire \cite{MR1233821}) in $\ZFC$, it is not a true ``dichotomy'' as II can easily have strategies for playing into both $\A$ and $\A^c$.

One consequence of Lemma \ref{lem:gen_blocks_Borel_strat} and the proof of Lemma \ref{lem:uB_strat} is that, given $(p^+)$-filter $\LF$ and a universally Baire set $\A\subseteq\bb^\infty(E)$, there is always an $X\in\LF$ such that one of $\A$ or $\A^c$ contains an $F_{\sigma\delta}$ set $\preceq$-dense below $X$.

We can now complete the proofs of Theorems \ref{thm:gen_over_L(R)} and \ref{thm:local_Rosendal_L(R)}.

\begin{proof}[Proof of Theorem \ref{thm:gen_over_L(R)}.]
	We have already proven the $(\Rightarrow)$ direction in Lemma \ref{lem:forcing_family}. For the remaining direction, let $\D\subseteq\bb^\infty(E)$ be a $\preceq$-dense open set which is in $\L(\R)$, and thus universally Baire by Theorem \ref{thm:FMW_UBsets}. By Lemma \ref{lem:uB_strat}, there is an $X\in\LF$ such that II has a strategy in $G[X]$ for playing into either $\D$ or $\D^c$. By Lemma \ref{lem:no_strat_into_dense_comp}, the latter can never occur. Thus, II has a strategy in $G[X]$ for playing into $\D$. Since $\LF$ is strategic, there is a play by this strategy, say $Z$, with $Z\in\D\cap\LF\neq\emptyset$.
\end{proof}

\begin{lemma}\label{lem:L(R)_filters_Rosendal}
	Assume that there is a supercompact cardinal. Let $\LF\subseteq\bb^\infty(E)$ be a strategic $(p^+)$-filter. Every subset of $\bb^\infty(E)$ in $\L(\R)$ is $\LF$-strategically Ramsey.	
\end{lemma}

\begin{proof}
	Let $\A\subseteq\bb^\infty(E)$ be in $\L(\R)$, and fix $\vec{x}\in\bb^{<\infty}(E)$ and $X\in\LF$. By Theorem \ref{thm:L(R)_strat_Ramsey}, the set of all $Y\preceq X$ witnessing that $\A$ is strategically Ramsey is $\preceq$-dense below $X$, and is clearly in $\L(\R)$. Since $\LF$ is $\L(\R)$-generic, $\LF$ must contain such a $Y$.	
\end{proof}

\begin{proof}[Proof of Theorem \ref{thm:local_Rosendal_L(R)}.]
	Let $\A\subseteq\bb^\infty(E)$ be in $\L(\R)$, and fix $\vec{x}\in\bb^{<\infty}(E)$ and $X\in\LH$. Let $\LG$ be $\V$-generic for $(\LH,\preceq^*)$ and contain $X$. By Lemma \ref{lem:forcing_family}, $\LG$ is a strategic $(p^+)$-filter in $\V[\LG]$. By Lemma \ref{lem:L(R)_filters_Rosendal}, there is a $Y\in\LG\restr X$ witnessing that $\A$ is strategically Ramsey in $\V[\LG]$. Since forcing with $(\LH,\preceq^*)$ adds no new reals, $Y$ witnesses that $\A$ is $\LH$-strategically Ramsey in $\V$.
\end{proof}

\section{Normed spaces and a local Gowers dichotomy}\label{sec:8}

We now consider the case when $E$ is a countably infinite-dimensional normed vector space, with \emph{normalized} basis $(e_n)$ (that is, $\|e_n\|=1$ for all $n$), over a countable subfield $F$ of $\C$ so that the norm takes values in $F$. If $V$ is a subspace of $E$, let $S(V)=\{x\in V:\|x\|=1\}$.

Let $\bb^{\infty}_1(E)=\{(x_n)\in\bb^\infty(E):\forall n(\|x_n\|= 1)\}$ and $\bb^{<\infty}_1(E)=\{\vec{x}\in\bb^{<\infty}(E):\forall n<|\vec{x}|(\|x_n\|=1)\}$. For $X\in\bb^\infty(E)$, let $[X]=\{Y\in\bb^\infty_1(E):Y\preceq X\}$. Taking $E$ discrete, $\bb^\infty_1(E)$ is a closed subset of the Polish space $\bb^\infty(E)$, thus itself Polish.

For $X=(x_n),Y=(y_n)\in\bb^\infty_1(E)$ and $\Delta=(\delta_n)$ a sequence of positive real numbers, written $\Delta>0$, we write $d(X,Y)\leq\Delta$ if for all $n$, $\|x_n-y_n\|\leq\delta_n$. Given $\A\subseteq\bb^\infty_1(E)$ and $\Delta>0$, let
\[
	\A_\Delta = \{Y\in\bb^\infty_1(E):\exists X\in\A (d(X,Y)\leq\Delta)\},
\]
the \emph{$\Delta$-expansion} of $\A$. We collect a few useful properties of $\Delta$-expansions in a lemma which will be used tacitly in what follows. The proof is left to the reader.

\begin{lemma}\label{lem:Delta_exp}
	Let $\A\subseteq\bb^\infty_1(E)$ and $\Delta>0$.
	\begin{enumerate}
		\item If $\A=\bigcup_{i\in I}\A_i$, then $\A_\Delta=\bigcup_{i\in I}(\A_i)_\Delta$.
		\item If $\A$ is analytic, then so is $\A_\Delta$.
		\item $(\A_\Delta)^c\subseteq((\A_\Delta)^c)_\Delta\subseteq\A^c$.
		\item If $0<\Gamma\leq\Delta/2$, then $((\A_\Delta)^c)_\Gamma\subseteq(\A_\Gamma)^c$.
	\end{enumerate}\qed
\end{lemma}

The notions of \emph{family}, \emph{filter}, \emph{fullness}, \emph{$(p)$-property}, etc, in $\bb^\infty_1(E)$, are defined exactly as for $\bb^\infty(E)$ in \S\ref{sec:2}. Moreover, all of the results established in the previous sections could have been carried out in $\bb^\infty_1(E)$ in the event that $E$ is normed. The only necessary modification is that in the games $G[\vec{x},X]$ and $F[\vec{x},X]$, the two players must play normalized block sequences and vectors, respectively. This will be assumed in what follows.

For $D\subseteq S(E)$, let
\[
	D_\eps=\{x\in S(E):\exists y\in D(\|x-y\|\leq\eps)\}.
\]
We weaken the notion of fullness to the following approximate version.\footnote{While this hampers our ability to reuse results of \S\ref{sec:3} and \S\ref{sec:7}, we hope that it will enable further applications. An elementary proof of Proposition \ref{prop:almost_full_full}, without the hypothesis of being ``strategic'', would greatly simplify the situation in the cases of interest.}
\begin{defn}
	A family $\LH\subseteq\bb_1^\infty(E)$ is \emph{almost full} if whenever $D\subseteq S(E)$ and $X\in\LH$ are such that $D$ is $\LH$-dense below $X$ (that is, for all $Y\in\LH\restr X$, there is a $Z\preceq Y$ with $S(\langle Z\rangle)\subseteq D$), then for any $\eps>0$, there is a $Z\in\LH\restr X$ with $S(\langle Z\rangle)\subseteq D_\eps$.
\end{defn}

\begin{defn}
	If a family has the $(p)$-property and is almost full we call it a \emph{$(p^*)$-family}. Likewise for \emph{$(p^*)$-filter}, \emph{strategic $(p^*)$-family}, etc. 
\end{defn}

The following is a discrete version of Gowers weakly Ramsey property \cite{MR1954235}, relativized to a family $\LH$.

\begin{defn}
	Given a family $\LH\subseteq\bb^\infty_1(E)$, a set $\A\subseteq\bb^\infty_1(E)$ is \emph{$\LH$-weakly Ramsey} if for every $\Delta>0$ and $X\in\LH$, there is a $Y\in\LH\restr X$ such that either
	\begin{enumerate}[label=\rm{(\roman*)}]
		\item $[Y]\subseteq\A^c$, or
		\item II has a strategy in $G[Y]$ for playing into $\A_\Delta$.
	\end{enumerate}
\end{defn}

The first goal of this section is to show that for certain $(p^*)$-families $\LH$, analytic sets in $\bb_1^\infty(E)$ are $\LH$-weakly Ramsey. We begin with variants of Lemmas \ref{lem:good_worse} and \ref{lem:local_Rosendal_open}, and Theorem \ref{thm:local_Rosendal}, for $(p^*)$-families. Since dealing with both families and $\Delta$-expansions requires some care, we include proofs of these results. As in \S\ref{sec:3}, they are very similar to those in \cite{MR2604856}.

\begin{defn}
	Given a family $\LH\subseteq\bb^\infty_1(E)$, $\A\subseteq\bb^\infty_1(H)$ and $\Delta>0$, for $\vec{y}\in\bb_1^{<\infty}(E)$ and $Y\in\LH$, we say the pair $(\vec{y},Y)$ is \emph{$\Delta$-good}/\emph{$\Delta$-bad}/\emph{$\Delta$-worse} if it is good/bad/worse for the set $\A_\Delta$ (in the sense of Definition \ref{def:good_bad_worse}). Further:
	\begin{enumerate}[label=\rm{(\arabic*)}]
		\item $(\vec{y},Y)$ is \emph{$\Delta^*$-good} if it is $\Delta(|\vec{y}|)$-good,
		\item $(\vec{y},Y)$ is \emph{$\Delta^*$-bad} if it is $\Delta(|\vec{y}|)$-bad,
		\item $(\vec{y},Y)$ is \emph{$\Delta^*$-worse} if it is $\Delta^*$-bad and there is a $n$ such that for all $v\in S(\langle Y/n\rangle)$, $(\vec{y}\concat v,Y)$ is $\Delta^*$-bad.
	\end{enumerate}
	Here, $\Delta(m)=(\delta_0/2,\delta_1/2,\ldots,\delta_{m-1}/2,\delta_m,\delta_{m+1},\ldots)$.
\end{defn}
 
Note that $\Delta^*$-good implies $\Delta$-good and $\Delta^*$-bad implies $\Delta/2$-bad.

\begin{lemma}\label{lem:good*_worse*}
	If $\LH$ is a $(p^*)$-family and $\A\subseteq\bb^\infty(E)$, then for every $\vec{x}\in\bb^{<\infty}(E)$, $X\in\LH$ and $\Delta>0$, there is a $Y\in\LH\restr X$ such that either
	\begin{enumerate}[label=\rm{(\roman*)}]
		\item $(\vec{x},Y)$ is $\Delta$-good, or
		\item I has a strategy in $F[\vec{x},Y]$ for playing into
		\[
			\{(z_n):\forall n(\vec{x}\concat(z_0,\ldots,z_n),Y) \text{ is $\Delta/2$-bad}\}.
		\]	
	\end{enumerate}
\end{lemma}

\begin{proof}
	Let $\LH$, $\A$, $X\in\LH$ and $\Delta>0$ be given. As in the proof of Lemma \ref{lem:good_worse}, for any $\vec{y}$ and $\Gamma>0$, the set
	\[
		\LD_{\vec{y}}^\Gamma=\{Y:(\vec{y},Y)\text{ is $\Gamma$-good or $\Gamma$-bad}\}
	\]	
	is $\preceq$-dense open in $\LH$, and if $(\vec{y},Y)$ is $\Gamma$-bad, then for every $V\in\LH\restr Y$, there is a $Z\preceq V$ such that for all $x\in S(\langle Z\rangle)$, $(\vec{y}\concat x,Y)$ is not $\Gamma$-good.
	
	\begin{claim}
		For any $\vec{y}\in \bb^{<\infty}_1(E)$, the set
		\[
			\LE_{\vec{y}}=\{Y:(\vec{y},Y)\text{ is $\Delta^*$-good or $\Delta^*$-worse}\}
		\]
		is $\preceq$-dense open in $\LH$.	
	\end{claim}
	
	\begin{proofclaim}
		Let $Y\in\LH$. By diagonalizing over the sets $\LD_{\vec{z}}^{\Delta(|\vec{z}|)}$, we may assume that for all $\vec{z}$, $(\vec{z},Y)$ is $\Delta^*$-good or $\Delta^*$-bad. Assume that $(\vec{y},Y)$ is $\Delta^*$-bad. Let 
		\[
			D=\{x\in S(E):(\vec{y}\concat x,Y)\text{ is not $\Delta(|\vec{y}|)$-good}\}.
		\]
		Take $\eps=\delta_{|\vec{y}|}/2$. By almost fullness, there is a $Z\in\LH\restr Y$ such that $S(\langle Z\rangle)\subseteq D_\eps$. Given $z\in S(\langle Z\rangle)$, pick $z'\in D$ with $\|z-z'\|<\eps$. If $(\vec{y}\concat z,Z)$ is $\Delta^*$-good, then there is a strategy $\alpha$ for II in $G[\vec{y}\concat z,Z]$ for playing into $\A_{\Delta(|\vec{y}|+1)}$. We may assume that all plays according to $\alpha$ are above $z$ and $z'$, so we can treat $\alpha$ as a strategy $\alpha'$ for II in $G[\vec{y}\concat z',Z]$. If $\vec{y}\concat {z'}\concat W$ is an outcome of $\alpha'$, then $\vec{y}\concat z\concat W$ is an outcome of $\alpha$, and thus in $\A_{\Delta(|\vec{y}|+1)}$. By our choice of $\eps$, it follows that $\vec{y}\concat {z'}\concat W$ is in $\A_{\Delta(|\vec{y}|)}$. Then, $(\vec{y}\concat z',Z)$ is $\Delta(|\vec{y}|)$-good, contradicting that $z'\in D$. Thus, $(\vec{y}\concat z,Z)$ is $\Delta^*$-bad, and $(\vec{y},Z)$ is $\Delta^*$-worse.
	\end{proofclaim}
	
	Returning to the proof of the lemma, assume that $\vec{x}=\emptyset$. By the claim, we can find $Y\in\LH\restr X$ such that for all $\vec{y}$, $(\vec{y},Y)$ is either $\Delta^*$-good or $\Delta^*$-worse. If $(\emptyset, Y)$ is $\Delta^*$-good, we're done, so assume that it is $\Delta^*$-worse. In this case, we define a strategy for I in $F[Y]$ for playing into $\{(z_n):\forall n(z_0,\ldots,z_n,Y) \text{ is $\Delta^*$-worse}\}$ exactly as in the proof of Lemma \ref{lem:good_worse}.
\end{proof}

\begin{lemma}[cf.~Lemma 2 in \cite{MR2604856}]\label{lem:local_Rosendal_open*}
	Let $\LH\subseteq\bb^\infty_1(E)$ be a $(p^*)$-family. Given $\A\subseteq\bb^\infty_1(E)$ open, $x\in\bb^{<\infty}_1(E)$, $X\in\LH$, and $\Delta>0$, there is a $Y\in\LH\restr Y$ such that either
	\begin{enumerate}[label=\rm{(\roman*)}]
		\item I has a strategy in $F[\vec{x},Y]$ for playing into $(\A_{\Delta/2})^c$, or
		\item II has a strategy in $G[\vec{x},Y]$ for playing into $\A_\Delta$.
	\end{enumerate}
\end{lemma}

\begin{proof}
	The proof is similar to Lemma \ref{lem:local_Rosendal_open}, using Lemma \ref{lem:good*_worse*}.
\end{proof}

\begin{lemma}[cf.~Lemma 4 in \cite{MR2604856}]\label{lem:unions_Rosendal*}
	Let $\LH\subseteq\bb^\infty_1(E)$ be a $(p^*)$-family. Suppose that $\A=\bigcup_{n\in\N}\A_n$, each $\A_n\subseteq \bb^\infty_1(E)$. Let $\vec{x}$, $X\in\LH$, and $\Delta>0$ be given. Then, there is a $Y\in\LH\restr X$ such that either
	\begin{enumerate}[label=\rm{(\roman*)}]
		\item I has a strategy in $F[\vec{x},Y]$ for playing into $(\A_{\Delta/2})^c$, or
		\item II has a strategy in $G[Y]$ for playing into
		\begin{align*}
			\{(z_k):\exists n\forall V\in\LH\restr Y(&\text{I has no strategy in $F[\vec{x}\concat(z_0,\ldots,z_n),V]$}\\
			 &\text{for playing into $((\A_n)_\Delta)^c$})\}.
		\end{align*}
	\end{enumerate}
\end{lemma}

\begin{proof}
	For $Y\in\LH$, $\vec{y}\in\bb^{<\infty}(E)$, and $n\in\N$, we say $(\vec{y},n)$ \emph{$\Gamma$-accepts} $Y$ if I has a strategy in $F[\vec{y},Y]$ for playing into $((\A_n)_\Gamma)^c$ and $(\vec{y},n)$ \emph{$\Gamma$-rejects} $Y$ if for all $Z\in\LH\restr Y$, $(\vec{y},n)$ does not $\Gamma$-accept $Z$. Both acceptance and rejection are $\preceq^*$-hereditary in $\LH$, and the sets
	\[
		\LD^\Gamma_{\vec{y},n}=\{Y:(\vec{y},n)\text{ $\Gamma$-accepts or $\Gamma$-rejects } Y\}
	\]
	are clearly $\preceq$-dense open in $\LH$. By the $(p)$-property, we can find $Y\in\LH\restr X$ such that for all $\vec{y}$ and $n$, $(\vec{y},n)$ either $\Delta/2$-accepts or $\Delta/2$-rejects $Y$. Put
	\[
		R=\{(z_k):\exists n(\vec{x}\concat(z_0,\ldots,z_n),n)\text{ $\Delta/2$-rejects } Y\},
	\]
	and notice that $R$ is open in $\bb^\infty_1(E)$. By Lemma \ref{lem:local_Rosendal_open*}, there is $Y'\in\LH\restr Y$ such that either II has a strategy in $G[Y']$ for playing into $R_{\Delta/2}$, or I has a strategy in $F[Y']$ for playing into $(R_{\Delta/4})^c\subseteq R^c$. In the first case, suppose that $(z_k)$ is an outcome of II's strategy. Then, there is $(z_k')$ with $\|z_k-z_k'\|\leq\delta_k/2$ for all $k$, and an $n$ such that $(\vec{x}\concat(z_0',\ldots,z_n'),n)$ $\Delta/2$-rejects $Y$. We claim $(\vec{x}\concat(z_0,\ldots,z_n),n)$ $\Delta$-rejects $Y$. If not, then for some $Z\in\LH\restr Y$, I has a strategy in $F[\vec{x}\concat(z_0,\ldots,z_n),Z]$ for playing into $((\A_n)_{\Delta})^c$. This yields a strategy for I in $F[\vec{x}\concat(z_0',\ldots,z_n'),Z]$ for playing into $(((\A_n)_{\Delta})^c)_{\Delta/2}$. By Lemma \ref{lem:Delta_exp}(d), $(((\A_n)_{\Delta})^c)_{\Delta/2}\subseteq((\A_n)_{\Delta/2})^c$, and so $(\vec{x}\concat(z_0',\ldots,z_n'),n)$ fails to $\Delta/2$-reject $Y$, a contradiction. Thus, $(z_k)$ is as desired for (ii).
	
Suppose that I has a strategy $\sigma$ in $F[Y]$ for playing into $(R_{\Delta/4})^c\subseteq R^c$. In particular, I plays $(z_k)$ such that for all $n$, I has a strategy $\sigma_{(z_0,\ldots,z_n)}$ in $F[\vec{x}\concat(z_0,\ldots,z_n),Y]$ to play into $((\A_n)_{\Delta/2})^c$. As in the proof of Lemma 4 in \cite{MR2604856}, we successively put more strategies for I into play, and obtain a strategy for playing into $\bigcap_n((\A_n)_{\Delta/2})^c=(\A_{\Delta/2})^c$.
\end{proof}

\begin{thm}[cf.~Theorem 5 in \cite{MR2604856}]\label{thm:local_Rosendal*}
	Let $\LH\subseteq\bb^\infty_1(E)$ be a $(p^*)$-family. If $\A\subseteq\bb^\infty_1(E)$ is analytic, $\Delta>0$, $\vec{x}\in\bb^{<\infty}_1(E)$, and $X\in\LH$, then there is a $Y\in\LH\restr Y$ such that either
	\begin{enumerate}[label=\rm{(\roman*)}]
		\item I has a strategy in $F[\vec{x},Y]$ for playing into $(\A_{\Delta/2})^c$, or
		\item II has a strategy in $G[\vec{x},Y]$ for playing into $\A_\Delta$.
	\end{enumerate}
\end{thm}

\begin{proof}
	We consider the case when $\vec{x}=\emptyset$. Let $F:\N^\N\to\A$ be a continuous surjection and for each $s\in\N^{<\N}$, let $\A_s=F''(N_s)$ where $N_s=\{\alpha\in\N^\N:s\subseteq\alpha\}$. Note that $\A_s=\bigcup_n\A_{s\concat n}$. 
	
	Let $R(s,\vec{x},Y)$ (for $Y\in\LH$) be the set of all $(z_k)$ for which there is an $n$ such that for all $Z\in\LH\restr Y$, I has no strategy in $F[\vec{x}\concat(z_0,\ldots,z_n),Z]$ for playing into $((\A_{s\concat n})_{\Delta})^c$. By Lemma \ref{lem:unions_Rosendal*} and the $(p)$-property, there is an $Y\in\LH\restr X$ such that for all $\vec{x}$ and $s\in\N^{<\N}$, either
	\begin{enumerate}[label=\rm{(\roman*)}]
		\item I has a strategy in $F[\vec{x},Y]$ for playing into $((\A_s)_{\Delta/2})^c$, or
		\item II has a strategy in $G[Y]$ for playing into $R(s,\vec{x},X)$.	
	\end{enumerate}
	
	Suppose I has no strategy in $F[Y]$ for playing into $(\A_{\Delta/2})^c=((\A_\emptyset)_{\Delta/2})^c$. We will describe a strategy for II in $G[Y]$ for playing into $\A_{\Delta}$: As II has a strategy in $G[Y]$ for playing into $R(\emptyset,\emptyset,Y)$, they follow this strategy until $(z_0,\ldots,z_{n_0})$ has been played such that I has no strategy in $F[(z_0,\ldots,z_{n_0}),Y]$ for playing into $((\A_{s\concat n_0})_{\Delta})^c$. By the assumption on $Y$, II must have a strategy in $G[Y]$ to play in $R((n_0),(z_0,\ldots,z_{n_0}),Y)$. II follows this until a further $(z_{n_0+1},\ldots,z_{n_0+n_1+1})$ has been played so that I has no strategy in $F[(z_0,\ldots,z_{n_0},\ldots,z_{n_0+n_1+1}),Y]$ for playing into $((\A_{{s}\concat {n_0}\concat n_1})_{\Delta})^c$. 
	
	We continue in this fashion, exactly as in the proof of Theorem 5 in \cite{MR2604856}, so that the outcome $Z=(z_n)$ satisfies that for all $k$, with $m_k=(\sum_{j\leq k}n_k)+k$, there is some $Z^k\sqsupseteq(z_0,\ldots,z_{m_k})$ in $(\A_{(n_0,\ldots,n_k)})_{\Delta}=(F''(N_{(n_0,\ldots,n_k)}))_{\Delta}$. Continuity of $F$ ensures that, for $\alpha=(n_0,n_1,\ldots)$, $d(F(\alpha),Z)\leq\Delta$.	
\end{proof}

The following result provides the link between strategically Ramsey sets and weakly Ramsey sets.

\begin{thm}[Rosendal \cite{MR2566964} \cite{MR2604856}]\label{thm:Rosendal_spread}
	Suppose that, for some $X\in\bb^\infty_1(E)$, I has a strategy in $F[X]$ to play into some set $\A\subseteq\bb^\infty_1(E)$. Then, for any $\Delta>0$, there is a sequence of finite intervals $I_0<I_1<\cdots$ in $\N$ such whenever $Y=(y_n)\preceq X$ and $\forall n\exists m(I_0<y_n<I_m<y_{n+1})$, we have that $Y\in\A_\Delta$.
\end{thm}

Inspired by this theorem, we define:

\begin{defn}
	A family $\LH\subseteq\bb^\infty_1(E)$ is \emph{spread} if whenever $X=(x_n)\in\LH$ and $I_0<I_1<\cdots$ is a sequence of intervals in $\N$, there is a $Y=(y_n)\in\LH\restr X$ such that $\forall n\exists m(I_0<y_n<I_m<y_{n+1})$.
\end{defn}

This property is analogous to the ``(q)-property'' (see Lemma 7.4 of \cite{MR2603812}) for coideals on $\N$: One can show that a coideal $\LH$ on $\N$ has the (q)-property if and only if for every $x\in\LH$ and sequence of finite intervals $I_0<I_1<\cdots$, there is a $y\in\LH\restr x$ such that $\forall n\exists m(I_0<y_n<I_m<y_{n+1})$.

By appropriately thinning down a block sequence, we see the following:

\begin{lemma}
	Given a sequence of intervals $I_0<I_1<\cdots$ in $\N$, the set
	\[
		\{(y_n):\forall n\exists m(I_0<y_n<I_m<y_{n+1})\}.
	\]
	is $\preceq$-dense open in $\bb^\infty_1(E)$.\qed
\end{lemma}

Clearly, $\bb^\infty_1(E)$ itself is spread. As in \S\ref{sec:5}, one can build spread filters (which are full, almost full, strategic, etc) under additional set-theoretic hypotheses or by forcing. We note that the strong $(p)$-property suffices:

\begin{lemma}\label{lem:strong_p_spread}
	If $\LH\subseteq\bb^\infty_1(E)$ is a strong $(p)$-family, then it is spread. In particular, strategic families are spread.
\end{lemma}

\begin{proof}
	Fix $X\in\LH$, and let $I_0<I_1<\cdots$ be an increasing sequence of intervals in $\N$. Consider the following strategy $\sigma$ for I in $F[X]$: $\sigma(\emptyset)=\max(I_0)$. If II responds with some $y_0>\sigma(\emptyset)$, then let $\sigma(y_0)=\max(I_m)$, where $I_m$ is the first interval entirely above $\supp(y_0)$. Continue in this fashion. Any outcome $(y_n)$ will satisfy $\forall n\exists m(I_0<y_n<I_m<y_{n+1})$. Since $\LH$ is a strong $(p)$-family, Theorem \ref{thm:strong_p_strat_I} implies that some outcome is in $\LH$.
\end{proof}

\begin{thm}\label{thm:local_Gowers_normed}
	Let $\LH\subseteq\bb^\infty_1(E)$ be a spread $(p^*)$-family. Then, every analytic set is $\LH$-weakly Ramsey.	
\end{thm}

\begin{proof}
	Let $\A\subseteq\bb_1^\infty(E)$ be analytic. Fix $X\in\LH$ and $\Delta>0$. By Theorem \ref{thm:local_Rosendal*}, there is $Y\in\LH\restr X$ such that either I has a strategy in $F[Y]$ for playing into $(\A_{\Delta/2})^c$, or II has a strategy in $G[Y]$ for playing into $\A_\Delta$. In the latter case, we're done, so assume the former. Theorem \ref{thm:Rosendal_spread} and $\LH$ being spread implies that there is some $Z\in\LH\restr Y$ with $[Z]\subseteq((\A_{\Delta/2})^c)_{\Delta/2}\subseteq\A^c$.
\end{proof}

In order to extend to sets in $\L(\R)$, we will use the following analogue of Lemma \ref{lem:uB_strat}.

\begin{lemma}\label{lem:uB_strat*}
	Let $\LF\subseteq\bb^\infty_1(E)$ be a $(p^*)$-filter. If $\A\subseteq\bb^\infty_1(E)$ is such that continuous images of $\A$ are universally Baire, then for any $X\in\LF$ and $\Delta>0$, there is a $Y\in\LF\restr X$ for which II has a strategy in $G[Y]$ for playing into one of $(\A_{\Delta/8})^c$ or $\A_\Delta$.
\end{lemma}

\begin{proof}
	Let $X\in\LF$ and $\Delta>0$. By Lemma \ref{lem:good*_worse*}, there is a $Y\in\LF\restr X$ such that either $(\emptyset, Y)$ is $\Delta$-good or I has a strategy $\sigma$ in $F[Y]$ for playing into
	\[
		\{(z_n):\forall n(z_0,\ldots,z_n,Y) \text{ is $\Delta/2$-bad}\}.
	\]
	In the former case, we're done, so assume the latter.
	
	By hypothesis, $\A_{\Gamma}$ is universally Baire for all $\Gamma$. In particular, we may let $\dot{\A}_{\Delta/4}$ be a $\P(\LF)$-name for $\A_{\Delta/4}$ and $\LD$ a countable collection of dense open subsets of $\P(\LF)$ such that
	\begin{enumerate}[label=\rm{(\roman*)}]
		\item $\{q\in\P(\LF):q\text{ decides } \dot{X}_\text{gen}\in\dot{A}\}$ is in $\LD$, and
		\item whenever $G$ is $\LD$-generic in $\P(\LF)$, $\dot{X}_\text{gen}$ is in $\bb^\infty_1(E)$ and $\dot{X}_\text{gen}(G)$ is in $\A_{\Delta/4}$ if and only if there is a $q\in G$ such that $q\forces_{\P(\LF)}\dot{X}_\text{gen}\in\dot{\A}_{\Delta/4}$.
	\end{enumerate}

	We claim that $(\emptyset, Y,\sigma)\forces_{\P(\LF)}\dot{X}_{\text{gen}}\notin\dot{\A}_{\Delta/4}$.
	
	Suppose not, then there is a $(\vec{y},Z,\tau)\leq(\emptyset,Y,\sigma)$, with $Z\in\LF$, such that $(\dot{y},Z,\tau)\forces_{\P(\LF)}\dot{X}_{\text{gen}}\in\dot{\A}_{\Delta/4}$. Applying Lemma \ref{lem:gen_blocks_Borel_strat}(b) and Theorem \ref{thm:local_Rosendal*}, there is a $W\in\LF\restr Z$ such that II has a strategy $\alpha$ in $G[\vec{y},W]$ for playing into $(\G_{\LD,(\vec{y},Z,\tau)})_{\Delta/4}$. As in the proof of Lemma \ref{lem:uB_strat}, $\G_{\LD,(\vec{y},Z,\tau)}\subseteq\A_{\Delta/4}$, so $\alpha$ is a strategy for II in $G[\vec{y},W]$ for playing into $\A_{\Delta/2}$. This, however, contradicts the fact that $\sigma$ ensures $(\vec{y},Z)$ is $\Delta/2$-bad.
	
	Thus, $(\emptyset, Y,\sigma)\forces_{\P(\LF)}\dot{X}_{\text{gen}}\notin\dot{\A}_{\Delta/4}$. But then, exactly as in the preceding paragraph, we may find $W\in\LF\restr Y$ such that II has a strategy in $G[W]$ for playing into $(\G_{\LD,(\emptyset,Y,\sigma)})_{\Delta/8}$, and thus into $((\A_{\Delta/4})^c)_{\Delta/8}\subseteq(\A_{\Delta/8})^c$, where the last containment follows from Lemma \ref{lem:Delta_exp}(d).
\end{proof}

In what follows, we strengthen the hypotheses on the basis $(e_n)$, asserting that there is some $K>0$ such that for all $m\leq n$ and scalars $(a_k)$,
\[
	\|\sum_{k\leq m} a_ke_k\| \leq K \|\sum_{k\leq n} a_k e_k\|.
\]
This is equivalent to $(e_n)$ being a Schauder basis of the completion $\bar{E}$ of $E$, cf.~Proposition 1.1.9 \cite{MR2192298}. The infimum of all such $K$ as above is called the \emph{basis constant} of $(e_n)$. The following Lemma about perturbations of blocks sequences appears to be well-known.

\begin{lemma}\label{lem:perturb_blocks}
	For any $\Delta>0$, there is a $\Gamma>0$ such that whenever $X=(x_n),X'=(x'_n)\in\bb^\infty_1(E)$ satisfy $d(X',X)\leq\Gamma$, then $[X']\subseteq [X]_\Delta$. In fact, if $Y'\in[X']$, then $\tilde{Y}\in[X]$ and $d(Y',\tilde{Y})\leq\Delta$, where $\tilde{Y}$ is the normalization of the image of $Y'$ under the linear map extending $x'_n\mapsto x_n$.
\end{lemma}

\begin{proof}
	Let $\Delta>0$. If $K$ is the basis constant of $(e_n)$, then by Lemma 1.3.5 in \cite{MR2192298}, the basis constant of $X$ is $\leq K$. Pick $\Gamma>0$ with $\sum_{n\geq m}\gamma_n\leq\min\{1/6K,\delta_m/8K\}$. For $X'=(x_n')$ with $d(X',X)\leq\Gamma$, consider the map on the completions $T:\bar{\langle X\rangle}\to\bar{\langle X'\rangle}$ extending $x_n\mapsto x_n'$. $T$ is a bounded linear isomorphism, as whenever $v=\sum a_nx_n\in\bar{\langle X\rangle}$,
	\begin{align*}
		\|Tv\|-\|v\|\leq\|Tv-v\|&\leq\|\sum a_nx_n'-\sum a_nx_n\|\leq\sup_n|a_n|\sum\|x_n'-x_n\|\\
		&\leq2K\|v\|\sum\|x_n'-x_n\|\leq1/3\|v\|,
	\end{align*}
	and so $\|T\|\leq 4/3$. Using $1/\|T^{-1}\|=\inf_{\|v\|=1}\|Tv\|$, we have $\|T^{-1}\|\leq3/2$.
	
	As the basis constant for $X'$ is also $\leq K$, for $v'=\sum_{n\geq m} a_nx'_n\in\bar{\langle X'\rangle}$, we have that $\|T^{-1}v'-v'\|\leq\delta_m/4\|v'\|$ by a similar argument as above.

	If $v'$ is a unit vector, then we also have that
	\[
		|1-\frac{1}{\|T^{-1}v'\|}|\leq\|T\|\|T^{-1}v'-v'\|\leq (4/3)(\delta_m/4)\leq\delta_m/3.
	\]
	For $Y'=(y_m')\in[X']$, we claim $d(Y',\tilde{Y})\leq\Delta$, where $\tilde{Y}$ is the normalization of $Y=(y_m)=(T^{-1}(y_m'))$. Observe that
	\begin{align*}
		\|y_m-\frac{1}{\|y_m\|}y_m\|\leq\left|1-\frac{1}{\|T^{-1}(y'_m)\|}\right|\|T^{-1}(y'_m)\|\leq(\delta_m/3)(3/2)=\delta_m/2.
	\end{align*}
	Thus, for all $m$, 
	\begin{align*}
		\|y'_m-\frac{1}{\|y_m\|}y_m\|&=\|y_m'-y_m\|+\|y_m-\frac{1}{\|y_m\|}y_m\| \leq\delta_m.
	\end{align*}
\end{proof}

The following lemma expresses the uniform continuity of the games $F[X]$ and $G[X]$.

\begin{lemma}\label{lem:cont_games}
	Let $\A\subseteq\bb^\infty_1(E)$ and $\Delta>0$. There is a $\Gamma>0$ such that whenever $X\in\bb^\infty_1(E)$ is such that I (II, respectively) has a strategy in $F[X]$ ($G[X]$, respectively) for playing into $\A$ and $d(X,X')\leq\Gamma$, then I (II, respectively) has a strategy in $F[X']$ ($G[X']$, respectively) for playing into $\A_\Delta$.
\end{lemma}

\begin{proof}
	Take $\Gamma>0$ as in Lemma \ref{lem:perturb_blocks}. Suppose I has a strategy $\sigma$ in $F[X]$ for playing into $\A$ and $d(X,X')\leq\Gamma$. We define a strategy $\sigma'$ for I in $F[X']$. Let $\sigma'(\emptyset)=\sigma(\emptyset)$. Inductively, suppose that $\sigma'(y'_0,\ldots,y'_k)$ has been defined and is equal to $\sigma(y_0,\ldots,y_k)$, where $y_0,\ldots,y_k$ is a valid play by II in $F[X]$ against $\sigma$, and $\|y'_i-y_i\|\leq\gamma_i$ for $0\leq i\leq k$. Suppose that $y'_{k+1}>\sigma'(y'_0,\ldots,y'_k)$ in $S(\langle X'\rangle)$. By our choice of $\Gamma$, there is a $y_{k+1}>\sigma'(y'_0,\ldots,y'_k)=\sigma(y_0,\ldots,y_k)$ in $S(\langle X\rangle)$ with $\|y'_{k+1}-y_{k+1}\|\leq\gamma_{k+1}$. Let $\sigma'(y'_0,\ldots,y'_k,y'_{k+1})=\sigma(y_0,\ldots,y_k,y_{k+1})$. It follows that $\sigma'$ is a strategy for playing into $\A_\Delta$.
	
Suppose that II has a strategy $\alpha$ in $G[X]$ for playing into $\A$, and $d(X,X')\leq\Gamma$. Let $T:\langle X\rangle\to\langle X'\rangle$ be as in the proof of Lemma \ref{lem:perturb_blocks}. We define a strategy $\alpha'$ for II in $G[X']$. Suppose that I begins by playing $Y'_0\in[X']$. Let $\alpha'(Y'_0)=\tilde{T}(\alpha(\tilde{T}^{-1}(Y'_0)))$, where $\tilde{T}$ and $\tilde{T^{-1}}$ indicate taking normalizations. Continue in this fashion. Then, $\alpha$ is a strategy for playing into $\A_\Delta$.
\end{proof}

\begin{thm}\label{thm:L(R)_filter*}
	Assume that there is a supercompact cardinal. Let $\LF\subseteq\bb^\infty_1(E)$ be a strategic $(p^*)$-filter. Then, every set $\A\subseteq\bb^\infty_1(E)$ in $\L(\R)$ is $\LF$-weakly Ramsey.
\end{thm}

\begin{proof}
	Let $\A\subseteq\bb^\infty_1(E)$ be in $\L(\R)$, $X\in\LF$, and $\Delta>0$. By Theorem \ref{thm:L(R)_strat_Ramsey}, the set $\D$ of all $Y\preceq X$ such that either I has a strategy in $F[Y]$ for playing into $(\A_{\Delta/2})^c$, or II has a strategy in $G[Y]$ for playing into $\A_{\Delta/2}$, is $\preceq$-dense open, and is clearly in $\L(\R)$. By Lemmas \ref{lem:no_strat_into_dense_comp} and \ref{lem:uB_strat*}, there is a $Y\in\LF\restr X$ such that II has a strategy for playing into $\D_\Gamma$, where $\Gamma$ is as in Lemma \ref{lem:cont_games}, applied to $\Delta/4$. Since $\LF$ is strategic, there is a $Z\in\LF\restr Y$ which is in $\D_\Gamma$. By our choice of $\Gamma$, then either I has a strategy in $F[Z]$ for playing into $((\A_{\Delta/2})^c)_{\Delta/4}\subseteq(\A_{\Delta/2})^c$, or II has a strategy in $G[Z]$ for playing into $\A_\Delta$. In the latter case, we're done, and in the former case, we need only apply Theorem \ref{thm:Rosendal_spread} and Lemma \ref{lem:strong_p_spread}.
\end{proof}

We will use the following analogue of Lemma \ref{lem:forcing_family}, whose proof is similar and left to the reader.

\begin{lemma}\label{lem:forcing_family*}
	For $\LH\subseteq\bb_1^\infty(E)$ a $(p^*)$-family, forcing with $(\LH,\preceq^*)$ adds no new reals and if $\LG\subseteq\LH$ is $\L(\R)$-generic for $(\LH,\preceq^*)$, $\LG$ will be a $(p^*)$-filter. If $\LH$ is strategic (spread, respectively), then $\LG$ will also be strategic (spread, respectively).
\end{lemma}

\begin{thm}\label{thm:local_Rosendal_L(R)*}
	Assume that there is a supercompact cardinal. Let $\LH\subseteq\bb^\infty_1(E)$ be a strategic $(p^*)$-family. Then, every set $\A\subseteq\bb^\infty_1(E)$ in $\L(\R)$ is $\LH$-weakly Ramsey.
\end{thm}

\begin{proof}
	The proof is similar to that of Theorem \ref{thm:local_Rosendal_L(R)}, using Lemma \ref{lem:forcing_family*} and Theorem \ref{thm:L(R)_filter*}.
\end{proof}

Some of the above can be simplified in the case when the family $\LH$ in question is \emph{invariant under small perturbations}, that is, there is some $\Delta>0$ so that $\LH_\Delta=\LH$. The reason lies in the following fact:

\begin{prop}\label{prop:almost_full_full}
	If $\LH$ is a strategic $(p^*)$-family which is invariant under small perturbations, then $\LH$ is a $(p^+)$-family as well.	
\end{prop}

\begin{proof}\hspace{-0.5em}\footnote{We suspect that an elementary proof of this result can be found and that ``strategic'' can be relaxed to ``spread''.}
	Let $D\subseteq S(E)$ be $\LH$-dense below some $X\in\LH$ and put $\D=\{Y\preceq X:S(\langle Y\rangle)\subseteq D\}$. Take $\Delta>0$ so that $\LH_\Delta=\LH$. Note that $\D$ is closed and thus it and its continuous images are universally Baire. Let $\LG$ be a $\V$-generic filter for $(\LH,\preceq^*)$ which contains $X$, so that by Lemma \ref{lem:forcing_family*}, $\LG$ is a strategic $(p^*)$-filter in $\V[\LG]$. By Lemma \ref{lem:uB_strat*} in $\V[\LG]$, there is a $Y\in\LG\restr X$ so that II has a strategy in $G[Y]$ for playing into one of $(\D_{\Delta/8})^c$ or $\D_\Delta$. However, as I has a strategy in $G[Y]$ for playing into $\D$, and $(\D_{\Delta/8})^c\subseteq\D^c$ by Lemma \ref{lem:Delta_exp}(c), II's strategy must be for playing into $\D_\Delta$. Since forcing with $(\LH,\preceq^*)$ added no new reals, such a strategy must exist in $\V$ (we are using Lemma \ref{lem:outcomes_strat_dense} implicitly here). As $\LH$ is strategic and $\LH_\Delta=\LH$, we have that $\LH\cap\D\neq\emptyset$, showing that $\LH$ is full.
\end{proof}

We now extend these principles to Banach spaces. In what follows, $B$ is a (separable) Banach space with normalized Schauder basis $(e_n)$. We say that a countable field $F$ is \emph{suitable} if the norm on $E_F$, the $F$-span of $E$, takes values in $F$. Let $\langle X\rangle_F$ the $F$-span of $X\in\bb^\infty(E_F)$. If $V$ is a subspace of $B$, let $S(V)=\{v\in V:\|v\|=1\}$.

Let $\bb^\infty_1(B)$ be the set of all infinite block sequences (with respect to $(e_n)$) in $B$, which we endow with the Polish topology inherited from $B^\N$. The relations $\preceq$ and $\preceq^*$ extend to $\bb^\infty_1(B)$. For $Y\in\bb^\infty_1(B)$, let $[Y]^*=\{Z\in\bb^\infty_1(B):Z\preceq Y\}$. We denote by $G^*[Y]$ the \emph{Gowers game} defined as before, except that the players may now play real (complex) block sequences and block vectors. The notions of \emph{family}, \emph{$(p)$-family}, \emph{spread}, and \emph{strategic} are defined as before, with appropriate modifications for real (complex) scalars.

Strategic families in $\bb^\infty_1(B)$ arise naturally from strategic families in $\bb^\infty_1(E_F)$: Given a strategic $\LH\subseteq\bb^\infty_1(E_F)$, if $\hat{\LH}$ is invariant under small perturbations and equal to the $\preceq$-upwards closure of $\LH_\Delta$ (taken in $\bb^\infty_1(B)$) for some small $\Delta>0$, then $\hat{\LH}$ is strategic. This follows from the fact that Lemma \ref{lem:perturb_blocks} and the proof of Lemma \ref{lem:cont_games} can be carried out in $B$.

\begin{defn}
	 We say that $\LH$ is \emph{almost full} if whenever $D\subseteq S(B)$ is closed and $\LH$-dense below some $X\in\LH$ (that is, for all $Y\in\LH\restr X$, there is a $Z\preceq Y$ with $S(\langle Z\rangle)\subseteq D$), then for any $\eps>0$, there is a $Y\in\LH\restr X$ with $S(\bar{\langle Y\rangle})\subseteq D_\eps$.
\end{defn}

\begin{defn}
	An almost full $(p)$-family in $\bb_1^\infty(B)$ is called a \emph{$(p^*)$-family}.
\end{defn}

While we have reused this terminology, the meaning should be clear from context. The following is the relativized version of Gowers weakly Ramsey property \cite{MR1954235}.

\begin{defn}
	Given a family $\LH\subseteq\bb^\infty_1(B)$, a set $\A\subseteq\bb^\infty_1(B)$ is \emph{$\LH$-weakly Ramsey} if for every $\Delta>0$ and $X\in\LH$, there is a $Y\in\LH\restr X$ such that either
	\begin{enumerate}[label=\rm{(\roman*)}]
		\item $[Y]^*\subseteq\A^c$, or
		\item II has a strategy in $G^*[Y]$ for playing into $\A_\Delta$.
	\end{enumerate}
\end{defn}

Proving Theorems \ref{thm:local_Gowers} and \ref{thm:local_Gowers_L(R)} amounts to showing that for spread (strategic) $(p^*)$-families $\LH\subseteq\bb_1^\infty(B)$ which are invariant under small perturbations, analytic ($\L(\R)$) sets are $\LH$-weakly Ramsey.

\begin{lemma}\label{lem:diag_pert_dense}
	Let $F$ be suitable. If $X_0\succeq X_1\succeq X_2\succeq\cdots$ is a $\preceq$-decreasing sequence in $\bb^\infty_1(E_F)$, $X\in\bb^\infty_1(B)$ is such that $X\preceq X_n$ for all $n$, and $\Delta>0$, then there is an $X'\in\bb^\infty_1(E_F)$ with $X'\in[X]_{\Delta}$, and $X'\preceq^* X_n$ for all $n$
\end{lemma}

\begin{proof}
	Let $(X_n)$, $X$ and $\Delta>0$ be as described, say with $X=(x_n)$. We construct $X'=(x_n')$ as follows: There is an $M_0\in\N$ so that $\bar{\langle X/M_0\rangle_F}\subseteq\bar{\langle X_0\rangle_F}$. Let $x_{n_0}$ be the first entry of $X/M_0$. Pick a unit vector $x_0'\in\langle X_0\rangle_F$ such that $d(x_{n_0},x_0')\leq\delta_0$. Continue inductively. At stage $k$, we have chosen $M_0<\cdots<M_k$ and $x_0'<\cdots<x_k'$ so that if $x_{n_i}$ is the first entry of $X/M_i$, then $x_i'\in\langle X_i\rangle_F$ and $d(x_{n_i},x_i')\leq\delta_i$, for $i\leq k$. By construction, $X'/n\preceq X_n$ for all $n$, and $X'\in[X]_{\Delta}$.
\end{proof}

\begin{lemma}\label{lem:p*_p+_Bspace}
	If $\LH\subseteq\bb^\infty_1(B)$ is a $(p^*)$-family which is invariant under small perturbations, then $\LH\cap\bb^\infty_1(E_F)$ is a $(p^*)$-family for any suitable subfield $F$ of $\R$ (or $\C$). If $\LH$ is spread (strategic, respectively), then so is $\LH\cap\bb^\infty_1(E_F)$.
\end{lemma}

\begin{proof}
	Let $\LH$ and $F$ be as described and put $\tilde{\LH}=\LH\cap\bb^\infty_1(E_F)$. Lemma \ref{lem:diag_pert_dense} implies that $\tilde{\LH}$ is a $(p)$-family. To see that $\tilde{\LH}$ is almost full, let $D\subseteq S(E_F)$ be $\tilde{\LH}$-dense below $X\in\tilde{\LH}$, and take $\eps>0$. Consider $\bar{D_{\eps/3}}\subseteq S(B)$. For $\Delta=(\eps/3,\eps/3,\ldots)$, let $\Gamma$ be as in Lemma \ref{lem:perturb_blocks}. For any $Y\in\LH\restr X$, there is a $Y'\in\tilde{\LH}\restr X$, with $d(Y,Y')\leq\Gamma$ and $Z'\preceq Y'$ with $S(\langle Z\rangle)\subseteq D$. By our choice of $\Gamma$, there is a $Z\in[Y]^*$ with $S(\langle Z\rangle)\subseteq D_{\eps/3}$, and so $S(\bar{\langle Z\rangle})\subseteq\bar{D_{\eps/3}}$. Thus, $\bar{D_{\eps/3}}$ is $\LH$-dense below $X$. By almost fullness of $\LH$, there is a $W\in\LH\restr X$ with $S(\bar{\langle W\rangle})\subseteq(\bar{D_{\eps/3}})_{\eps/3}$. Then, one can find a $W'\in\tilde{\LH}\restr X$ with $S(\langle W'\rangle)\subseteq D_\eps$, showing that $\tilde{\LH}$ is almost full.
	
	To see that $\LH$ being strategic implies that $\hat{\LH}$ is strategic, let $\alpha$ be a strategy for II in $G[X]$, with $X\in\hat{\LH}$. Define a strategy $\alpha'$ in $G^*[X]$ which is equal to $\alpha$ on their shared domain, and otherwise plays so that the outcomes are sufficiently small (using Lemma \ref{lem:perturb_blocks} and our assumption about $\LH$) perturbations of outcomes of $\alpha$. Then, if any outcome of $\alpha'$ is in $\LH$, an outcome of $\alpha$ must be in $\hat{\LH}$. The proof for being spread is left to the reader.
\end{proof}

\begin{proof}[Proof of Theorem \ref{thm:local_Gowers}.]
	Suppose that $\A\subseteq\bb^\infty_1(B)$ is analytic, $\Delta>0$, and $X\in\LH$ is such that for no $Y\in\LH\restr X$ is $[Y]^*\subseteq\A^c$. Let $F$ be a suitable field for $(e_n)$. Let $\tilde{\LH}=\LH\cap\bb^\infty_1(E_F)$. If there was some $Y\in\tilde{\LH}\restr X$ with $[Y]\subseteq(\A_{\Delta/3})^c\cap\bb^\infty_1(E_F)$, then $[Y]^*\subseteq((\A_{\Delta/3})^c)_{\Delta/3}\subseteq\A^c$, contrary to our assumption. Thus, by Lemma \ref{lem:p*_p+_Bspace} and Theorem \ref{thm:local_Gowers_normed}, there is a $Y\in\tilde{\LH}\restr X$ such that II has a strategy in $G[Y]$ for playing into $\A_{\Delta/2}\cap\bb^\infty_1(E_F)$. Easy perturbation arguments show that II has a strategy in $G^*[Y]$ for playing into $\A_\Delta$.
\end{proof}

\begin{proof}[Proof of Theorem \ref{thm:local_Gowers_L(R)}.]
	The proof is similar to that of 	Theorem \ref{thm:local_Gowers}, using Theorem \ref{thm:local_Rosendal_L(R)*}, or alternatively, Proposition \ref{prop:almost_full_full} and Theorem \ref{thm:local_Rosendal_L(R)}.
\end{proof}

The following is an analytical example of a strategic $(p^*)$-family, which, though trivial in the sense that is $\preceq$-downwards closed, we hope suggests further applications:

\begin{example}\label{ex:blocks_equiv_c0_lp}
	Given $B$ as above, suppose that $B$ contains a normalized block sequence $X$ equivalent to the standard basis of $c_0$ or $\ell^p$ for $1\leq p<\infty$. Let $\LH$ be the set of all block 	sequences in $B$ which have a further block subsequence equivalent to $X$. Then, $\LH$ is a strategic $(p^*)$-family which is invariant under small perturbations. These facts follows from the block homogeneity characterization of the standard bases of $c_0$ and $\ell^p$, Lemma 2.1.1 in \cite{MR2192298}.
\end{example}

\section{Projections in the Calkin algebra}\label{sec:9}

Given a Banach space with a Schauder basis, one might wish to develop a notion of forcing with block sequences ``modulo small perturbation'' and then prove an analog of Theorem \ref{thm:gen_over_L(R)}, characterizing $\L(\R)$-generic filters.\footnote{There are obstacles to this being a meaningful endeavor in general, e.g., in a hereditarily indecomposable Banach space, the collection of \emph{all} infinite-dimensional subspaces modulo small perturbations forms a filter, cf.~(iii) on p.~820 of \cite{MR1954235}, and is thus trivial as a forcing notion.} We focus on a particular variant of this which is of significant interest.

Let $H$ be a complex infinite-dimensional separable Hilbert space with orthonormal basis $(e_n)$. Note that any normalized block sequence (with respect to $(e_n)$) is necessarily orthonormal. Throughout, $E$ will denote the $\bar{\Q}$-linear span of $(e_n)$ in $H$, $\bb^\infty_1(E)$ the space of infinite normalized block sequences in $E$, and for $X\in\bb^\infty_1(E)$, $\langle X\rangle$ is the $\bar{\Q}$-span of $X$.

For $X\in\bb^\infty_1(E)$, let $P_X$ be the orthogonal projection onto $\bar{\langle X\rangle}$. Note that, for $X,Y\in\bb^\infty_1(E)$, $X\preceq Y$ if and only if $P_X\leq P_Y$ in the usual ordering of projections (that is, $P\leq Q$ if $\ran(P)\subseteq\ran(Q)$, or equivalently $PQ=P$). We call such projections \emph{block projections}.

Let $\LB(H)$ be the C*-algebra of bounded operators on $H$ and $\LK(H)$ the ideal of compact operators on $H$. The quotient $\LC(H)=\LB(H)/\LK(H)$ is also a C*-algebra, called the \emph{Calkin algebra}. We write $\pi:\LB(H)\to\LC(H)$ for the quotient map. 

Denote by $\LP(H)$ ($\LP_\infty(H)$, respectively) the set (infinite-rank, respectively) projections in $\LB(H)$, and $\LP(\LC(H))$ ($\LP(\LC(H))^+$, respectively) the set of (non-zero, respectively) projections, i.e., self-adjoint idempotents, in $\LC(H)$. By Proposition 3.1 in \cite{MR2300900}, $\LP(\LC(H))=\pi(\LP(H))$. The ordering $\leq$ on $\LP(\LC(H))$ is inherited from the ordering on $\LP(H)$.

\begin{defn}
	\begin{enumerate}
		\item For projections $P,Q\in\LP(H)$, we write $P\leq_\ess Q$ if $\pi(P)\leq\pi(Q)$ in $\LP(\LC(H))$ and $P\equiv_\ess Q$ if $\pi(P)=\pi(Q)$.
		\item For $X,Y\in\bb^\infty_1(E)$, we write $X\leq_\ess Y$ if $P_X\leq_\ess P_Y$ and $X\equiv_\ess Y$ if $P_X\equiv_\ess P_Y$.
	\end{enumerate}
\end{defn}

The last sentence of the following lemma requires a slight modification of the original proof and is left to the reader.

\begin{lemma}[Proposition 3.3 in \cite{MR2300900}]\label{lem:Weaver}
	For $P$ and $Q$ projections on $H$, the following are equivalent:
	\begin{enumerate}[label=\rm{(\roman*)}]
		\item $P\leq_\ess Q$.
		\item For every $\eps>0$, there is a finite-codimensional subspace $V$ of $\ran(P)$ such that every unit vector $v\in V$ satisfies $d(v,\ran(Q))\leq\eps$.
	\end{enumerate}
	In the event that $P=P_X$ and $Q=P_Y$ for $X,Y\in\bb^\infty_1(E)$, one can replace ``finite-codimensional subspace'' in \textup{(ii)} with ``tail subspace''.
\end{lemma}

The following lemma is well-known:

\begin{lemma}\label{lem:sumbl_pert}
	Suppose that $\Delta=(\delta_n)>0$ is summable and $P$ and $Q$ are projections on $H$ whose ranges have orthonormal bases $(x_n)$ and $(y_n)$ respectively. If for all $n$, $\|x_n-y_n\|\leq\delta_n$, then $P\equiv_\ess Q$.
\end{lemma}

\begin{proof}
	Assuming that for all $n$, $\|x_n-y_n\|\leq\delta_n$, we will show that $P\leq_\ess Q$. The result follows by symmetry. Let $\eps>0$ and choose an $N$ such that $\sum_{n\geq N}\delta_n\leq\eps$. Let $V=\bar{\langle(x_n)_{n\geq N}\rangle}$, a finite-codimensional subspace of $\ran(P)$. If $v\in V$ is a unit vector, say with $v=\sum_{n\geq N}a_n x_n$, then for $y=\sum_{n\geq N}a_n y_n\in\ran(Q)$, we have
	\[
		\|v-y\|=\|\sum_{n\geq N}a_n(x_n-y_n)\|\leq\sum_{n\geq N}\|x_n-y_n\|\leq\eps.
	\]
	The claim follows by Lemma \ref{lem:Weaver}.
\end{proof}

In particular, $\equiv_\ess$-invariant families in $\bb^\infty_1(E)$ or $\bb^\infty_1(H)$ are invariant under small perturbations. The following observation can be proved using Lemma \ref{lem:sumbl_pert} and standard manipulations with basic sequences (cf.~Proposition 1.3.10 in \cite{MR2192298}).

\begin{lemma}\label{lem:blocks_ess_dense}
	The set of block projections is dense in $(\LP_\infty(H),\leq_\ess)$.
\end{lemma}

It follows that $(\LP(\LC(H))^+,\leq)$, $(\LP_\infty(H),\leq_\ess)$, and $(\bb^\infty_1(E),\leq_\ess)$ are equivalent as notions of forcing. It is for this reason that we focus on $(\bb^\infty_1(E),\leq_\ess)$.

\begin{lemma}\label{lem:diag_ess_dense}
	If $X_0\succeq X_1\succeq X_2\succeq\cdots$ is a $\preceq$-decreasing sequence in $\bb^\infty_1(E)$ and $X\in\bb^\infty_1(E)$ is such that $X\leq_\ess X_n$ for all $n$, then there is an $X'\leq_\ess X$ such that $X'\preceq^* X_n$ for all $n$.	
\end{lemma}

\begin{proof}
	This can be proved using Lemmas \ref{lem:Weaver} and \ref{lem:sumbl_pert} in a way similar to Lemma \ref{lem:diag_pert_dense}.
\end{proof}

Clearly, any $\preceq$-dense subset of $\bb^\infty_1(E)$ is also $\leq_\ess$-dense. The following lemma is a converse to this.

\begin{lemma}\label{lem:ess_dense_block_dense}
	If $\LD\subseteq\bb^\infty_1(E)$ is $\leq_\ess$-dense open, then it is $\preceq$-dense open.	
\end{lemma}

\begin{proof}
	Suppose $\LD\subseteq\bb^\infty_1(E)$ is $\leq_\ess$-dense open. Given any $X\in\bb^\infty_1(E)$, there is a $Y\in\LD$ with $Y\leq_\ess X$. Applying Lemma \ref{lem:diag_ess_dense} (with $X_n=X$ for all $n$), there is a $Y'\leq_\ess Y$ with $Y'\preceq X$. Then, $Y'\in\LD$.
\end{proof}

We can now establish Theorem \ref{thm:Calkin_L(R)_gen}, an analog of Theorem \ref{thm:gen_over_L(R)} for projections in the Calkin algebra. We first prove a more general result.

\begin{thm}\label{thm:ess_L(R)_gen}
	\begin{enumerate}
		\item If $\LG$ is an $\L(\R)$-generic filter for $(\bb^\infty_1(E),\leq_\ess)$, then $\LG$ is a strategic $(p^+)$-family.
		\item Assume that there is a supercompact cardinal. If $\LG\subseteq\bb^\infty_1(E)$ is a strategic $(p^*)$-family which is also a $\leq_\ess$-filter, then $\LG$ is $\L(\R)$-generic for $(\bb_1^\infty(E),\leq_\ess)$.
	\end{enumerate}
\end{thm}

\begin{proof}
	(a) Let $\LG$ be as described. Clearly, it is a family. To see that it is full, suppose that $D\subseteq S(E)$ is $\LG$-dense below some $X\in\LG$. Let
	\[
		\LD_0=\{Z:\langle Z\rangle \subseteq D \text{ or } \forall V\preceq X(\langle V\rangle\subseteq D \rightarrow V\bot Z)\},
	\]
	where $\bot$ denotes incompatibility with respect to $\preceq$. $\LD_0$ is $\preceq$-dense open by Lemma \ref{lem:dense_sets_diag_full}, thus $\leq_\ess$-dense as well, and clearly in $\L(\R)$, so there is a $Z\in\LD_0\cap(\LG\restr X)$. Then, there is a $Z'\preceq Z\preceq X$ with $S(\langle Z'\rangle)\subseteq D$, so we have that $\langle Z\rangle\subseteq D$, showing that $\LG$ is full.
	
	To see that $\LG$ is a $(p)$-family, let $X_0\succeq X_1\succeq X_2\succeq\cdots$ in $\LG$. Let
	\[
		\LD_1 = \{Y : \forall n(Y\preceq^* X_n) \text{ or } \exists n(Y\bot_{\ess} X_n)\},
	\]
	where $\bot_{\ess}$ denotes incompatibility with respect to $\leq_\ess$. We want to show that $\LD_1$ is $\leq_\ess$-dense. The set
	\[
		\LD_1' = \{Y : \forall n(Y\leq_\ess X_n) \text{ or } \exists n(Y\bot_{\ess} X_n)\}
	\]
	is $\leq_\ess$-dense open. Then, given any $X$, we can find a $Y\in\LD_1'$ below $X$. If $Y\bot_\ess X_n$ for some $n$, we're done. Otherwise, $Y\leq_\ess X_n$ for all $n$, and we can apply Lemma \ref{lem:diag_ess_dense} to find a $Y'\leq_\ess Y$ with $Y\preceq^* X_n$ for all $n$. Such a $Y'$ is in $\LD_1$, verifying that this set is $\leq_\ess$-dense. As $\LD_1$ is in $\L(\R)$, $\LG\cap\LD_1\neq\emptyset$, and anything in this intersection must be a diagonalization of $(X_n)$. It is likewise easy to see that $\LG$ must be strategic.\\
	
	\noindent(b) Let $\LD\subseteq\bb^\infty_1(H)$ be $\leq_\ess$-dense open and in $\L(\R)$. By Lemma \ref{lem:ess_dense_block_dense}, $\LD$ is also $\preceq$-dense open. For $\Delta>0$ summable, $\LD_\Delta=\LD$ by Lemma \ref{lem:sumbl_pert}. Thus, by Theorem \ref{thm:local_Gowers_normed}, there is an $X\in\LH$ such II has a strategy for playing into $\LD$. Since $\LG$ is strategic, it follows that $\LG\cap\LD\neq\emptyset$.
\end{proof}

\begin{proof}[Proof of Theorem \ref{thm:Calkin_L(R)_gen}]
	The ($\Rightarrow$) direction is proved by a straightforward verification of the relevant sets being $\preceq$-dense open, thus $\leq_\ess$-dense by Lemma \ref{lem:ess_dense_block_dense}. The ($\Leftarrow$) direction follows from Theorem \ref{thm:ess_L(R)_gen}(b) or Theorem \ref{thm:local_Gowers_L(R)}.
\end{proof}

We conclude this section by describing a hoped-for application of our machinery and its limitations. A \emph{state} $\tau$ on $\LB(H)$ is a linear functional on $\LB(H)$ which is \emph{positive}, that is, $\tau(T^*T)\geq 0$ for all $T$, and satisfies $\tau(I)=1$, where $I$ is the identity operator. The set of states forms a weak*-compact convex subset of the dual of $\LB(H)$ and thus has extreme points, called \emph{pure states}. These definitions generalize to any unital C*-algebra, including $\LC(H)$.

A state on $\LB(H)$ is \emph{singular} if it vanishes on $\LK(H)$. Composing with the quotient map $\pi:\LB(H)\to\LC(H)$ yields a bijective correspondence between singular pure states on $\LB(H)$ and pure states on $\LC(H)$.

For any choice of orthonormal basis $(f_k)$ for $H$, and any ultrafilter $\LU$ on $\N$, the functional defined by $\tau_\LU(T)=\lim_{k\to\LU}\langle Tf_k,f_k\rangle$ is a pure state which is singular if and only if $\LU$ is non-principal (cf.~Theorem 4.21 and Example 6.1 in \cite{FarahAST}). Such pure states are said to be \emph{diagonalizable}. On an abelian C*-algebra, pure states coincide with characters, so the aforementioned $\tau_\LU$ restricts to a pure state on the atomic maximal abelian self-adjoint subalgebra (or \emph{masa}) generated by the rank-one projections corresponding to the $f_k$. The following problem asks to what extent this is true of all pure states:

\begin{problem}[Kadison--Singer \cite{MR0123922}]	Does every pure state on $\LB(H)$ restrict to a pure state on some (atomic or continuous) masa?
\end{problem}

Anderson conjectured that not only is the answer to this question ``yes'', but that every pure state is of the form $\tau_\LU$, for some choice of orthonormal basis $(f_k)$ and ultrafilter $\LU$:

\begin{conj}[Anderson \cite{MR672813}]
	Every pure state on $\LB(H)$ is diagonalizable.
\end{conj}

Akemann and Weaver \cite{MR2403097} showed that the above problem of Kadison and Singer has a negative answer, and thus Anderson's conjecture is false, assuming $\CH$. It remains an open question whether Anderson's conjecture is consistent with $\ZFC$.

By the recent positive solution \cite{MR3374963} to the Kadison--Singer problem regarding \emph{extensions} of pure states (which differs from the above), Anderson's conjecture is equivalent to saying that every pure state on $\LB(H)$ restricts to a pure state on some atomic masa.

Following \cite{MR3043037}, we say that a subset $\LF\subseteq\LP(\LC(H))^+$ is \emph{centered}\footnote{These were called \emph{quantum filters} by Farah and Weaver \cite{FarahAST}.} if for every finite subset of $\LF$ has a lower bound in $\LP(\LC(H))^+$. $\LF$ is \emph{linked} if every pair of elements in $\LF$ has a lower bound in $\LP(\LC(H))^+$. \emph{Maximal centered} has the obvious meaning. Similarly, we define \emph{$\leq_\ess$-centered}, \emph{$\leq_\ess$-linked}, and \emph{maximal $\leq_\ess$-centered} in $\bb^\infty_1(E)$.

\begin{thm}[Farah--Weaver, Theorem 6.42 in \cite{FarahAST}]
	There is a bijective correspondence between singular pure states $\tau$ on $\LB(H)$ and maximal centered subsets of $\LP(\LC(H))^+$ via $\tau \mapsto \LF_\tau=\{p\in\LP(\LC(H))^+:\tau(p)=1\}$.
\end{thm}

If $\LF=\LF_\tau$ as above and $\tau$ fails to restrict to a pure state on any atomic masa, we say that $\LF$ \emph{yields a counterexample to Anderson's conjecture}.

\begin{thm}[essentially Farah--Weaver, cf.~Theorem 6.46 in \cite{FarahAST}]
	If $\LG$ is $\V$-generic for $\LP(\LC(H))^+$, then $\LG$ is a maximal centered set which yields a counterexample to Anderson's conjecture.
\end{thm}

In fact, this result uses much less than full genericity, or even genericity over $\L(\R)$. By considering the complexity of the dense sets involved in the proof, we obtain Theorem \ref{thm:pure_states}:

\begin{proof}[Proof of Theorem \ref{thm:pure_states}.]
	Let $\LH\subseteq\bb^\infty_1(E)$ be spread $(p^*)$-family which is $\leq_\ess$-centered and $\hat{\LH}$ the upwards closure of $\pi(\LH)$ in $\LP(\LC(H))^+$. First, we claim that $\hat{\LH}$ is a maximal centered set. Clearly, $\hat{\LH}$ is centered. For maximality, let $p\in\LP(\LC(H))^+$ be such that $p$ is compatible with every finite subset of $\hat{\LH}$. Let $P\in\LP(H)$ be such that $\pi(P)=p$, and define
	\[
		\LD_P = \{X:P_X\leq_\ess P \text{ or } P_X\bot_\ess P\},
	\]
	which is a co-analytic and $\leq_\ess$-dense open subset of $\bb^\infty_1(H)$. By Lemma \ref{lem:ess_dense_block_dense}, $\LD_P$ is $\preceq$-dense open, so by Theorem \ref{thm:local_Gowers_normed}, we can find a $Y\in\LH\restr X$ with $Y\in\LD_P$. It must then be the case that $P_Y\leq_\ess P$ and so $p\in\hat{\LH}$.
	
	To see that $\hat{\LH}$ yields a counterexample to Anderson's conjecture, we refer to the proof of Theorem 6.46 in \cite{FarahAST} and omit the details except to note that it suffices to show that $\LH$ meets the $\leq_\ess$-dense open sets 
	\[
		\LD_{\vec{J}}=\{X\in\bb^\infty_1(E):\forall n(\|P^{(f_k)}_{J_n\cup J_{n+1}} P_X\|<1/2)\},
	\]
	where $\vec{J}=(J_n)$ is a partition of $\N$ into finite intervals $J_n$ and $P^{(f_k)}_{J}$ denotes the orthogonal projection onto $\bar{\linspan}\{f_k:k\in J\}$, for $(f_k)$ an orthonormal basis of $H$. These sets are easily seen to be Borel, and meeting them with $\LH$ uses the combination of Lemma \ref{lem:ess_dense_block_dense} and Theorem \ref{thm:local_Gowers_normed} as before.
\end{proof}

For spread $(p^*)$-families, being $\leq_\ess$-linked implies being a $\leq_\ess$-filter:

\begin{lemma}\label{lem:linked_filter}
	Let $\LH\subseteq\bb^\infty_1(E)$ be a spread $(p^*)$-family which is, moreover, $\leq_\ess$-linked. Then, $\LH$ is a $\leq_\ess$-filter.
\end{lemma}

\begin{proof}
	Let $X,Y\in\LH$, and consider the set
	\[
		\LD=\{Z:(Z\leq_\ess X \text{ and } Z\leq_\ess Y) \text{ or } (Z \bot_\ess X \text{ or } Z\bot_\ess Y)\}.
	\]
	It is easy to check that $\LD$ is co-analytic. Clearly $\LD$ is $\leq_\ess$-dense open, thus $\preceq$-dense open by Lemma \ref{lem:ess_dense_block_dense}. By Theorem \ref{thm:local_Gowers_normed} applied to the analytic set $\A=\LD^c$, there is a $Z\in\LH$ with $[Z]_1\subseteq\LD$. In particular, $Z\in\LD$. Since $\LD$ is $\leq_\ess$-linked, we must have that $Z\leq_\ess X$ and $Z\leq_\ess Y$.
\end{proof}

By Lemma \ref{lem:linked_filter}, the maximal centered sets in Theorem \ref{thm:pure_states} are also filters in $\LP(\LC(H))^+$. The following result of Bice, using Shelah's model without $p$-points (VI.~\S4 in \cite{MR1623206}), presents an obstacle to $\ZFC$ constructions.

\begin{thm}[Bice \cite{MR3043037}]
	It is consistent with $\ZFC$ that no maximal centered set in $\LP(\LC(H))^+$ is a filter.	
\end{thm}

Consequently, we have:

\begin{cor}
	It is consistent with $\ZFC$ that no spread $(p^*)$-family in $\bb^\infty_1(E)$ can be $\leq_\ess$-linked, and in particular, that there are no spread $(p^*)$-filters.
\end{cor}

\section{Further questions}\label{sec:last}

Despite our constructions, under additional hypotheses, of $(p^+)$-filters, there remains a lack of examples of interesting, purely analytical, $(p^+)$ and $(p^*)$-families, Example \ref{ex:blocks_equiv_c0_lp} notwithstanding.

\begin{ques}
	Are there naturally occurring non-trivial ($\ZFC$) examples of $(p^+)$ or $(p^*)$ families of block sequences?
\end{ques}

While Theorem \ref{thm:Calkin_L(R)_gen} does give a criterion for $\L(\R)$-genericity for filters of projections in the Calkin algebra, it would be desirable to have a such criterion expressed in the language of C*-algebras.

\begin{ques}
	Can the (local) Ramsey theory of block sequences in a separable infinite-dimensional Hilbert space be described in C*-algebraic terms? Under large cardinals, is there a C*-algebraic characterization of $\L(\R)$-generic filters in the projections in the Calkin algebra?
\end{ques}

Lastly, as the sufficient conditions described in Theorem \ref{thm:pure_states} for producing a counterexample to Anderson's conjecture cannot be satisfied in Shelah's model without $p$-points, the status of Anderson's conjecture in that model appears to be a natural test question.

\begin{ques}
	Does Anderson's conjecture hold in Shelah's model without $p$-points?\footnote{Added in proof: This question has a negative answer. It follows from the fact that Anderson's conjecture fails whenever $\d$ is less than or equal to the analogue of $\t$ for $\LP(\LC(H))$ \cite{FarahAST}, which holds in Shelah's model.}
\end{ques}

\bibliographystyle{abbrv}

\begin{thebibliography}{10}

\bibitem{MR2403097}
C.~Akemann and N.~Weaver.
\newblock {$\mathcal{B}(H)$} has a pure state that is not multiplicative on any
  masa.
\newblock {\em Proc. Natl. Acad. Sci. USA}, 105(14):5313--5314, 2008.

\bibitem{MR2192298}
F.~Albiac and N.~J. Kalton.
\newblock {\em Topics in {B}anach space theory}, volume 233 of {\em Graduate
  Texts in Mathematics}.
\newblock Springer, New York, 2006.

\bibitem{MR672813}
J.~Anderson.
\newblock A conjecture concerning the pure states of {$\mathcal{B}(H)$} and a
  related theorem.
\newblock In {\em Topics in modern operator theory
  ({T}imi{\c{s}}oara/{H}erculane, 1980)}, volume~2 of {\em Operator Theory:
  Adv. Appl.}, pages 27--43. Birkh{\"a}user, Basel-Boston, Mass., 1981.

\bibitem{MR2145246}
S.~A. Argyros and S.~Todor{\v{c}}evi{{\'c}}.
\newblock {\em Ramsey methods in analysis}.
\newblock Advanced Courses in Mathematics. CRM Barcelona. Birkh{\"a}user
  Verlag, Basel, 2005.

\bibitem{MR1839387}
J.~Bagaria and J.~L{{\'o}}pez-Abad.
\newblock Weakly {R}amsey sets in {B}anach spaces.
\newblock {\em Adv. Math.}, 160(2):133--174, 2001.

\bibitem{MR1873008}
J.~Bagaria and J.~L{{\'o}}pez-Abad.
\newblock Determinacy and weakly {R}amsey sets in {B}anach spaces.
\newblock {\em Trans. Amer. Math. Soc.}, 354(4):1327--1349 (electronic), 2002.

\bibitem{MR1350295}
T.~Bartoszy{{\'n}}ski and H.~Judah.
\newblock {\em Set theory: {O}n the structure of the real line}.
\newblock A K Peters, Ltd., Wellesley, MA, 1995.

\bibitem{MR3043037}
T.~M. Bice.
\newblock Filters in {$\rm C^*$}-algebras.
\newblock {\em Canad. J. Math.}, 65(3):485--509, 2013.

\bibitem{MR891244}
A.~Blass.
\newblock Ultrafilters related to {H}indman's finite-unions theorem and its
  extensions.
\newblock In {\em Logic and combinatorics ({A}rcata, {C}alif., 1985)},
  volume~65 of {\em Contemp. Math.}, pages 89--124. Amer. Math. Soc.,
  Providence, RI, 1987.

\bibitem{MR2768685}
A.~Blass.
\newblock Combinatorial cardinal characteristics of the continuum.
\newblock In {\em Handbook of set theory. {V}ol. 1}, pages 395--489. Springer,
  Dordrecht, 2010.

\bibitem{MR906807}
A.~Blass and N.~Hindman.
\newblock On strongly summable ultrafilters and union ultrafilters.
\newblock {\em Trans. Amer. Math. Soc.}, 304(1):83--97, 1987.

\bibitem{local_Ramsey}
C.~A. Di~Prisco, J.~G. Mijares, and J.~Nieto.
\newblock Local ramsey theory. {A}n abstract approach.
\newblock To appear in {\em MLQ Math. Log. Q.}

\bibitem{MR1978344}
C.~A. Di~Prisco and S.~Todor{\v{c}}evi{{\'c}}.
\newblock Souslin partitions of products of finite sets.
\newblock {\em Adv. Math.}, 176(1):145--173, 2003.

\bibitem{MR1644345}
I.~Farah.
\newblock Semiselective coideals.
\newblock {\em Mathematika}, 45(1):79--103, 1998.

\bibitem{FarahAST}
I.~Farah and E.~Wofsey.
\newblock Set theory and operator algebras.
\newblock In J.~Cummings and E.~Schimmerling, editors, {\em Appalachian Set
  Theory 2006-2012}, number 406 in Lecture Notes Series, pages 63--120. London
  Mathematical Society, 2012.

\bibitem{MR1233821}
Q.~Feng, M.~Magidor, and W.~H. Woodin.
\newblock Universally {B}aire sets of reals.
\newblock In {\em Set theory of the continuum ({B}erkeley, {CA}, 1989)},
  volume~26 of {\em Math. Sci. Res. Inst. Publ.}, pages 203--242. Springer, New
  York, 1992.

\bibitem{MR2145797}
V.~Ferenczi and C.~Rosendal.
\newblock Ergodic {B}anach spaces.
\newblock {\em Adv. Math.}, 195(1):259--282, 2005.

\bibitem{MR1421876}
W.~T. Gowers.
\newblock A new dichotomy for {B}anach spaces.
\newblock {\em Geom. Funct. Anal.}, 6(6):1083--1093, 1996.

\bibitem{MR1954235}
W.~T. Gowers.
\newblock An infinite {R}amsey theorem and some {B}anach-space dichotomies.
\newblock {\em Ann. of Math. (2)}, 156(3):797--833, 2002.

\bibitem{MR0349574}
N.~Hindman.
\newblock Finite sums from sequences within cells of a partition of
  {$\mathbb{N}$}.
\newblock {\em J. Combinatorial Theory Ser. A}, 17:1--11, 1974.

\bibitem{MR0123922}
R.~V. Kadison and I.~M. Singer.
\newblock Extensions of pure states.
\newblock {\em Amer. J. Math.}, 81:383--400, 1959.

\bibitem{MR1321597}
A.~S. Kechris.
\newblock {\em Classical descriptive set theory}, volume 156 of {\em Graduate
  Texts in Mathematics}.
\newblock Springer-Verlag, New York, 1995.

\bibitem{MR1324462}
R.~A. Komorowski and N.~Tomczak-Jaegermann.
\newblock Banach spaces without local unconditional structure.
\newblock {\em Israel J. Math.}, 89(1-3):205--226, 1995.

\bibitem{MR0427070}
K.~Kunen.
\newblock Some points in {$\beta N$}.
\newblock {\em Math. Proc. Cambridge Philos. Soc.}, 80(3):385--398, 1976.

\bibitem{MR996504}
C.~Laflamme.
\newblock Forcing with filters and complete combinatorics.
\newblock{\em Ann. Pure Appl. Logic}, 42(2), 125--163, 1989.

\bibitem{MR2737185}
C.~Laflamme, L.~Nguyen Van~Th\'e, M.~Pouzet, and N.~Sauer.
\newblock Partitions and indivisibility properties of countable dimensional
  vector spaces.
\newblock {\em J. Combin. Theory Ser. A}, 118(1):67--77, 2011.

\bibitem{MR2069032}
P.~B. Larson.
\newblock {\em The stationary tower: Notes on a course by W. Hugh Woodin},
  volume~32 of {\em University Lecture Series}.
\newblock Amer. Math. Soc., Providence, RI, 2004.

\bibitem{MR2179777}
J.~L{{\'o}}pez-Abad.
\newblock Coding into {R}amsey sets.
\newblock {\em Math. Ann.}, 332(4):775--794, 2005.

\bibitem{MR3374963}
A.~W. Marcus, D.~A. Spielman, and N.~Srivastava.
\newblock Interlacing families {II}: {M}ixed characteristic polynomials and the
  {K}adison-{S}inger problem.
\newblock {\em Ann. of Math. (2)}, 182(1):327--350, 2015.

\bibitem{MR0491197}
A.~R.~D. Mathias.
\newblock Happy families.
\newblock {\em Ann. Math. Logic}, 12(1):59--111, 1977.

\bibitem{MR2330595}
J.~G. Mijares.
\newblock A notion of selective ultrafilter corresponding to topological
  {R}amsey spaces.
\newblock {\em MLQ Math. Log. Q.}, 53(3):255--267, 2007.

\bibitem{MR0373906}
K.~R. Milliken.
\newblock Ramsey's theorem with sums or unions.
\newblock {\em J. Combinatorial Theory Ser. A}, 18:276--290, 1975.

\bibitem{MoForcing}
J.~T. Moore.
\newblock The method of forcing.
\newblock In preperation, 2016.

\bibitem{MR2566964}
C.~Rosendal.
\newblock Infinite asymptotic games.
\newblock {\em Ann. Inst. Fourier (Grenoble)}, 59(4):1359--1384, 2009.

\bibitem{MR2604856}
C.~Rosendal.
\newblock An exact {R}amsey principle for block sequences.
\newblock {\em Collect. Math.}, 61(1):25--36, 2010.

\bibitem{MR1623206}
S.~Shelah.
\newblock {\em Proper and improper forcing}.
\newblock Perspectives in Mathematical Logic. Springer-Verlag, Berlin, second
  edition, 1998.

\bibitem{MR1074499}
S.~Shelah and W.~H. Woodin.
\newblock Large cardinals imply that every reasonably definable set of reals is
  {L}ebesgue measurable.
\newblock {\em Israel J. Math.}, 70(3):381--394, 1990.

\bibitem{MR0332480}
J.~Silver.
\newblock Every analytic set is {R}amsey.
\newblock {\em J. Symbolic Logic}, 35:60--64, 1970.

\bibitem{Smythe_thesis}
I.~B. Smythe.
\newblock {\em Set theory in infinite-dimensional vector spaces}.
\newblock PhD thesis, Cornell University, 2017.

\bibitem{MR2603812}
S.~Todor{\v{c}}evi{{\'c}}.
\newblock {\em Introduction to {R}amsey spaces}, volume 174 of {\em Annals of
  Mathematics Studies}.
\newblock Princeton University Press, Princeton, NJ, 2010.

\bibitem{MR2300900}
N.~Weaver.
\newblock Set theory and {$C^*$}-algebras.
\newblock {\em Bull. Symbolic Logic}, 13(1):1--20, 2007.

\bibitem{yyz_sel_fin}
Y.~Y. Zheng.
\newblock Selective ultrafilters on $\mathrm{FIN}$.
\newblock {\em Proc. Amer. Math. Soc.}, 145(12):5071--5086, 2017.

\end{thebibliography}

\end{document}